\RequirePackage{fix-cm}
\RequirePackage{amsmath}
\documentclass[smallextended,envcountsect]{svjour3hack}
\usepackage{mathptmx}       
\usepackage{helvet}         
\usepackage{courier}        
\usepackage{type1cm}        

\smartqed
\usepackage{graphicx}

\usepackage{amsmath,amsthm,amssymb}
\usepackage[dvipsnames]{xcolor}
\usepackage{float}
\usepackage{enumitem}
\usepackage[utf8]{inputenc}
\usepackage[T3,T1]{fontenc}
\DeclareSymbolFont{tipa}{T3}{cmr}{m}{n}
\DeclareMathAccent{\invbreve}{\mathalpha}{tipa}{16}
\usepackage{bm}
\usepackage[hypertexnames=false,colorlinks=true,breaklinks=true,bookmarks=true,urlcolor=blue,citecolor=blue,linkcolor=blue,bookmarksopen=false,draft=false]{hyperref}

\DeclareMathOperator{\vol}{vol}
\DeclareMathOperator{\convcl}{convcl}
\DeclareMathOperator{\conv}{conv}
\DeclareMathOperator{\adet}{adet}

\makeatletter
\let\c@proposition\c@theorem
\let\c@corollary\c@theorem
\let\c@lemma\c@theorem
\let\c@definition\c@theorem
\let\c@example\c@theorem
\let\c@conjecture\c@theorem
\let\c@conj\c@theorem
\makeatother

\allowdisplaybreaks

\journalname{JOTA}

\begin{document}

\title{Gaining or Losing Perspective for\\  Piecewise-Linear  Under-Estimators of\\ Convex Univariate Functions\footnote{A short preliminary version of some of this work will appear in the proceedings of CTW 2020: \hfil\break \phantom{xx} \url{https://sites.google.com/site/jonleewebpage/PLperspec_CTW_final.pdf}}}
\titlerunning{Perspective for Piecewise-Linear Under-Estimators of Convex Univariate Functions}

\author{Jon Lee \and Daphne Skipper \and Emily Speakman \and Luze Xu}%
\institute{J. Lee \at University of Michigan. \email{jonxlee@umich.edu}
\and D. Skipper \at United States Naval Academy. \email{skipper@usna.edu}
\and E. Speakman \at University of Colorado, Denver. \email{emily.speakman@ucdenver.edu}
\and L. Xu \at University of Michigan. \email{xuluze@umich.edu}}

\authorrunning{Lee, Skipper, Speakman, Xu}

\date{\today}

\maketitle

\begin{abstract}
We study MINLO (mixed-integer nonlinear optimization) formulations of the disjunction
$x\in\{0\}\cup[\ell,u]$, where $z$ is a binary indicator
of $x\in[\ell,u]$ ($0 \leq \ell <u$), and $y$ ``captures'' $f(x)$, which is assumed to be
convex and positive on its domain  $[\ell,u]$, but otherwise $y=0$ when $x=0$.
This model is very useful in nonlinear combinatorial optimization,
where there is a fixed cost of operating an activity at level $x$ in the operating range $[\ell,u]$, and then there is a further (convex) variable cost $f(x)$.
 In particular, we study relaxations related to
the perspective transformation of a natural piecewise-linear under-estimator of $f$, obtained
by choosing linearization points for $f$.
Using 3-d volume (in $(x,y,z)$) as a measure of the tightness of a convex relaxation, we
investigate relaxation quality as a function of $f$, $\ell$, $u$, and the linearization points chosen.
We make a detailed investigation for convex power functions $f(x):=x^p$, $p>1$.
\end{abstract}


\section{Introduction}

\subsection{Definitions and background}
Let $f$ be a univariate convex function with domain
$[\ell,u]$, where $0\leq \ell<u$. We assume that
$f$ is positive on $[\ell,u]$.
We are interested in  the mathematical-optimization context of modeling a function, represented by a variable $y$,
that is equal to a given convex function $f(x)$ on an ``operating range'' $[\ell,u]$
and equal to $0$ at $0$. We do this
using a $0/1$ indicator variable $z$ (which conveniently
allows for incorporating a fixed cost for $x$ being in the operating range),
and we represent the relevant set disjunctively
as follows.
We define
 \begin{align*}
 \invbreve{D}_f(\ell,u):=
&
\{(0,0,0)\} {\textstyle \bigcup}
\left\{ (x,y,1) \in \mathbb{R}^3 ~\negthinspace\negthinspace:~\negthinspace\negthinspace\negthinspace\negthinspace
\vphantom{\scriptstyle\frac{f(u)-f(\ell)}{u-\ell}}
\right.
\left.
 f(\ell) + {\scriptstyle \frac{f(u)-f(\ell)}{u-\ell} }(x-\ell)
\geq y \geq f(x),~ u\geq x \geq \ell
\right\}.
\end{align*}
Notice that for $x\in\{\ell,u\}$, we have $y=f(x)$. So, the upper bound on $y$
enables us to capture the convex hull of the graph of the convex $f(x)$ on $[\ell,u]$,
in the $z=1$ plane.

Next, following the notation of \cite{Perspec2019}, we define the \emph{perspective relaxation}
\begin{align*}
&\invbreve{S}^*_f(\ell,u) := \convcl \left\{ (x,y,z) \in \mathbb{R}^3 ~:~
\left(f(\ell)-  {\scriptstyle \frac{f(u)-f(\ell)}{u-\ell}} \ell\right)z + {\scriptstyle \frac{f(u)-f(\ell)}{u-\ell}} x
\geq y \geq z f(x/z),~ \right.\\
&\left. uz\geq x
\geq \ell z,~ 1\geq z > 0,~ y\geq 0
\vphantom{\scriptstyle\frac{f(u)-f(\ell)}{u-\ell}}
\right\},
\end{align*}
where $\convcl$ denotes the convex closure operator. 
Notice that ``perspectivizing'' the convex $f(x)$ produces a more complicated but still convex function
$z f(x/z)$, and handling such a function pushes us into the realm of conic programming.
On the other side, perspectivizing the (univariate) linear upper bound on $y$ leads to a (bivariate but still)
linear upper bound on $y$.
Intersecting $\invbreve{S}^*_f(\ell,u)$ with the hyperplane defined by $z=0$, leaves the
single point $(x,y,z)=(0,0,0)$, which is only in the set after we take the closure.
 In this way, the ``perspective and convex closure'' construction
gives us exactly the value $y=0$ that we want at $x=0$. Moreover, $\invbreve{S}^*_f(\ell,u)$
is precisely the convex closure of $\invbreve{D}_f(\ell,u)$.

We compare convex bodies relaxing
$\invbreve{S}^*_f(\ell,u)$
 via their
volumes, with an eye toward weighing the relative tightness
of relaxations against the difficulty of solving them.
Generally, working with $\invbreve{S}^*_f(\ell,u)$ implies
using a cone solver (e.g., Mosek), while relaxations
imply the possibility of using more general NLP or even LP solvers;
see \cite{Perspec2019} for more discussion on this important
motivating subject.
One key relaxation previously studied
requires that the domain of $f$ is all of $[0,u]$,
$f$ is convex on $[0,u]$, $f(0)=0$,
and $f$ is  increasing on $[0,u]$.
For example, convex power functions $f(x):=x^p$ with $p>1$ have these properties.
Assuming these properties, we define the \emph{na\"{\i}ve relaxation}
\begin{align*}
&\invbreve{S}^0_f(\ell,u):=
\left\{ (x,y,z) \in \mathbb{R}^3 ~:~
\left(f(\ell)-  {\scriptstyle \frac{f(u)-f(\ell)}{u-\ell}} \ell\right)z
  + {\scriptstyle \frac{f(u)-f(\ell)}{u-\ell}} x
\geq y \geq f(x),~ \right.\\
&\left. uz\geq x \geq  \ell z,~   1\geq z \geq 0
\vphantom{\scriptstyle\frac{f(u)-f(\ell)}{u-\ell}}
\right\}.
\end{align*}
While the na\"{\i}ve relaxation is weaker than the perspective relaxation,
it can be handled more efficiently and by a wider class of solvers
because of its simpler form involving $f(x)$ rather than $zf(x/z)$.

\subsection{Relation to previous literature}

The perspective transformation of a convex function is
well known in mathematics (see \cite{perspecbook}, for example).
Applying it in the context of our disjunction is
also well studied (see \cite{gunlind1,Frangioni2006,Akturk}, with applications
to nonlinear facility location and also mean-variance portfolio optimization in the
style of Markowitz).
The idea of using volume to compare relaxations was
introduced by \cite{LM1994} (also see \cite{LeeSkipperSpeakmanMPB2018}
and the references therein). Recently, \cite{PerspectiveWCGO,Perspec2019}
applied the idea of using volumes to evaluate and compare
the perspective relaxation with other relaxations of our disjunction.

Piecewise linearization is a very well studied and useful concept
for handling nonlinearities
(see, for example, \cite{CharnesLemke,LeeWilson} and also
the more recent \cite{Viel,TorViel}
and the many references therein).
It is a natural idea to strengthen a convex piecewise linearization
of a convex univariate function
using the perspective idea, and then to evaluate it using
volume computation. This is what we pursue here, concentrating on
piecewise-linear under-estimators of univariate convex functions.
We also wish to mention and emphasize that our techniques are
directly relevant for (additively) separable convex functions
(see \cite{Bonami,Berenguel}, and of course all of the exact global-optimization
solvers (which induce a lot of separability via reformulation using additional variables).

\subsection{Our contribution and organization}

Our focus is on
 relaxations related to natural piecewise-linear
under-estimators of $f$. Piecewise linearization is a standard method
for efficiently handling nonlinearities in optimization. For a convex
function, it is easy to get a piecewise-linear under-estimator.
But there are a few issues to consider: the number of
linearization points, how to choose them, and how to
handle the resulting piecewise-linearization.

In particular, we look at the behavior of the perspective relaxation associated with
a natural piecewise-linear under-estimator of a convex univariate function,
as we vary the placement and the number of linearization points
describing the piecewise-linear under-estimator.

In \S\ref{sec:PLest}, we introduce notation for a natural piecewise-linear under-estimator $g$ of $f$ on $[\ell,u]$, using linearizations of $f(x)$ at
$n+1(\geq 2)$ values of $x$, namely $\ell =: \xi_0 < \xi_1 < \cdots < \xi_n := u$,
we define the convex relaxation $\invbreve{U}^*_f(\bm\xi):=\invbreve{S}^*_g(\ell,u)$, and  we describe an efficient algorithm for
determining its volume (Theorem \ref{PLperspecvolformula} and Corollary \ref{cor:lineartime}).
Armed with this efficient algorithm, any global-optimization software
could decide between members of this family of
 formulations (depending on the number and placement of linearization points)
 and also alternatives (e.g., $\invbreve{S}^*_f(\ell,u)$
and $\invbreve{S}^0_f(\ell,u)$, explored in \cite{Perspec2019}), trading off tightness of the formulations against the relative ease/difficulty of working with them computationally.

In \S\ref{sec:xtothep}, we give a more detailed analysis for convex power functions $f(x):=x^p$, for $p>1$.
In \S\ref{sec:xtothep:quadratics}, focusing on quadratics ($p=2$), we solve the volume-minimization problem for $\vol(\invbreve{U}^*_f(\bm\xi))$  when $p=2$ (Theorem \ref{thm:p2n}), for an arbitrary number of linearization points, thus finding the
optimal placement of linearization points for convex quadratics.
Further, from this, we recover the associated formula from \cite{Perspec2019} for $\vol(\invbreve{S}^*_f(\ell,u))$ (Corollary \ref{cor:S2}), and we
demonstrate that the minimum volume is always less than the volume of the na\"{i}ve relaxation when $p=2$ (Corollary \ref{rem:compare}).
In \S\ref{sec:xtothep:nonquadratics}, focusing on non-quadratics ($p\not=2$), we first demonstrate with Theorem \ref{thm:localmin}  that
all stationary points  are strict local minimizer. Next, with Theorem \ref{conj:pn},
we demonstrate that for $p\leq 2$, that the volume function is strictly convex, and so in this
case (Corollary \ref{thm:partial_uniquemin}), we can conclude that it has a unique minimizer.
We establish that this also holds for $p>2$ (Theorem \ref{thm:unique_stationary} and Corollary \ref{thm:remaining_uniquemin}). We also establish that the optimal location of each linearization point is increasing in $p$ on $(1,\infty)$
(Theorem \ref{conj:increasing}). Finally, we establish a nice monotone behavior for Newton's method on our volume minimization problem (Theorem
\ref{thm:Newtonn}).
In \S\ref{sec:xtothep:single}, we consider optimal placement of a single non-boundary linearization point.
Furthermore, via a simple transformation, for the tricky case of
minimizing $\vol(\invbreve{U}^*_p(\ell,\xi_1,u))$ when $p>2$, we
can reduce that problem to maximizing a strictly concave function (Theorem \ref{thm:logconcave}).
Next, we provide some bounds on the minimizing $\xi_1$ (Theorem \ref{conj:location}).
This can be useful on determining a reasonable initial point for
a minimization algorithm or even for a reasonable static
rule for selecting  linearization points.
Next, we establish how good our bounds are in the case of $\ell=0$ (Proposition \ref{prop:deltabehavior0}).

In \S\ref{sec:lighter},
we consider several related relaxations that are less computational burdensome than the perspective relaxation
applied to a convex power function or even to a piecewise-linear under-estimator.
To demonstrate the type of results that can be established, we focus on convex power functions and
ultimately quadratics with equally-spaced linearization points. In particular, we establish how many linearization
points are needed for various approximations.


\section{Piecewise-linear under-estimation and perspective}\label{sec:PLest}

Piecewise-linear estimation is widely used in optimization.
\cite{LeeWilson} provides some key relaxations
using integer variables, even for non-convex functions
on  multidimensional (polyhedral) domains.
We are particularly interested in piecewise-linear \emph{under-}estimation
because of its value in global optimization.

Given convex $f:[\ell,u]\rightarrow \mathbb{R}_{++}$, we consider linearization points
\[
\ell =: \xi_0 < \xi_1 < \cdots < \xi_n := u
\]
in the domain of $f$, and we assume that $f$ is differentiable
at these $\xi_i$.

At each  $\xi_i$, we have the tangent line
\begin{equation}
y=f(\xi_i)+f'(\xi_i)(x-\xi_i), \label{eq:tangentline}\tag{$T_i$}
\end{equation}
for $i=0,\ldots,n$. Considering tangent lines \ref{eq:tangentline}
and $T_{i-1}$ (for adjacent points),
 we have the intersection point
\begin{equation}
(x,y):=(\tau_i,~ f(\xi_i)+f'(\xi_i)(\tau_i-\xi_i)),  \hbox{ for $i=1,\ldots,n$},  \label{eqint:point}\tag{$P_i$}
\end{equation}
where
\[
\tau_i :=
\frac{
\left[f(\xi_i)-f'(\xi_i)\xi_i  \right]  -\left[ f(\xi_{i-1})-f'(\xi_{i-1})\xi_{i-1} \right]
}
{
f'(\xi_{i-1})- f'(\xi_i)
}.
\]
Finally, we define
\begin{equation}
(x,y):=(\tau_0:=\ell,f(\ell)) \label{eqint:point0}\tag{$P_0$}
\end{equation}
and
\begin{equation}
(x,y):=(\tau_{n+1}:=u,f(u)). \label{eqint:pointn+1}\tag{$P_{n+1}$}
\end{equation}

It is easy to see that
$
\ell =: \tau_0 < \tau_1 < \cdots < \tau_{n+1} := u,
$
and that
 the piecewise-linear function $g:[\ell,u]\rightarrow \mathbb{R}$,
defined as the function having the graph that connects the \ref{eqint:point},
for $i=0,1,\ldots,n+1$, is a convex under-estimator of $f$ (agreeing with
$f$ at the $\xi_i$; 
see Fig. \ref{fig:pl}.
In what follows, $g$ is always defined
as above (from $f$ and   $\bm\xi$).
\begin{figure}[H]
  \centering
 \includegraphics[scale = 0.35]{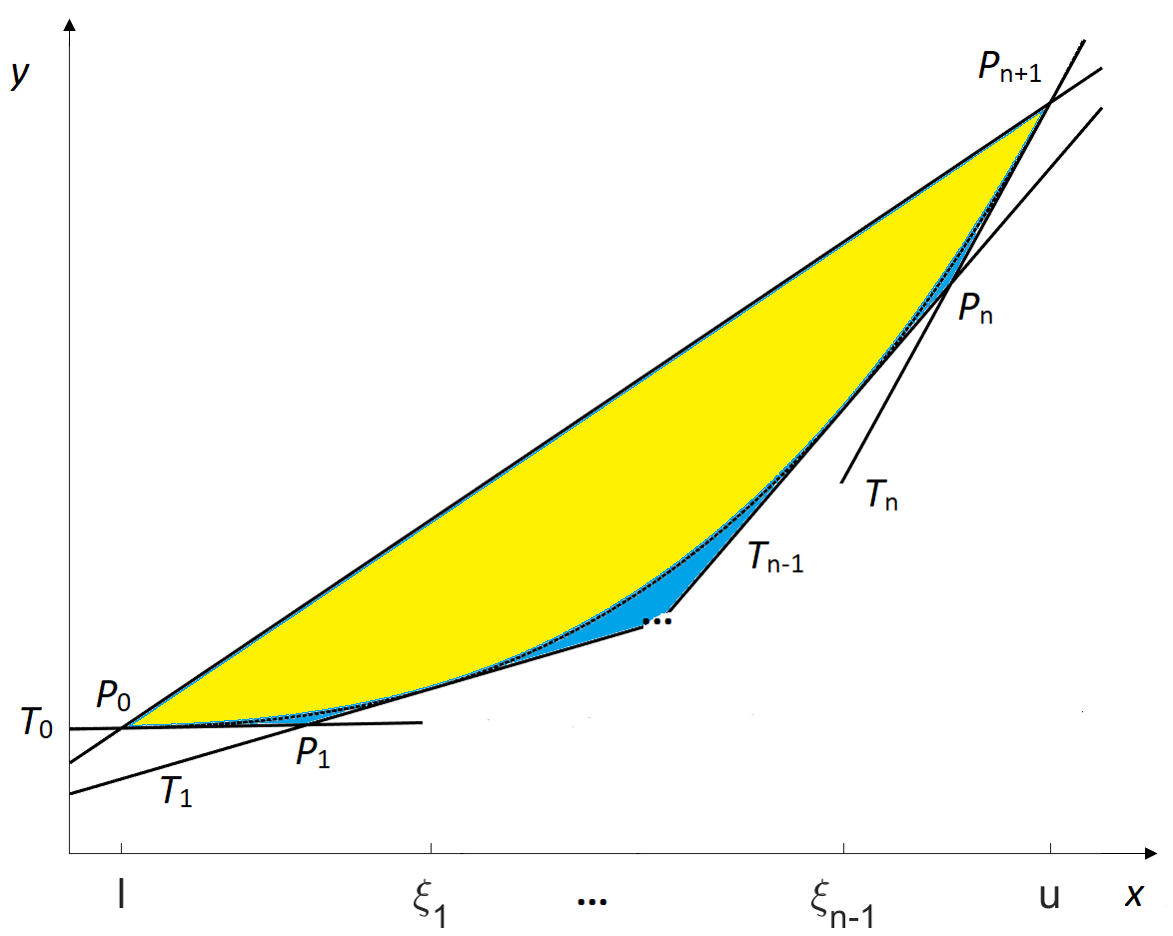}
  \caption{Piecewise-linear under-estimator}\label{fig:pl}
\end{figure}
We wish to compute the volume of the set $\invbreve{U}^*_f(\bm\xi):=\invbreve{S}^*_g(\ell,u)$.
To proceed, we work with the sequence $\tau_0,\tau_1,\ldots,\tau_{n+1}$ defined above.
Below and later, $\adet$ denotes the absolute value of the determinant.
\begin{theorem}\label{PLperspecvolformula}
\[
\vol(\invbreve{U}^*_f(\bm\xi))=\frac{1}{6}\sum_{i=1}^n
\adet\left(
  \begin{array}{ccc}
    \tau_0 & \tau_i & \tau_{i+1} \\
    g(\tau_0 ) & g(\tau_i) & g(\tau_{i+1}) \\
    1 & 1  & 1
  \end{array}
\right).
\]
\end{theorem}
\begin{proof}
We wish to compute the volume of the set $\invbreve{U}^*_f(\bm\xi)$.
This set is a pyramid with apex $(x,y,z)=(0,0,0)$ and base
equal to the intersection of $\invbreve{U}^*_f(\bm\xi)$ with the
hyperplane defined by the equation $z=1$. The height
of the apex over the base is unity. So  the volume of $\invbreve{U}^*_f(\bm\xi)$
is simply the area of the base divided by 3.
We will compute the area of the base by straightforward 2-d triangulation.
Our triangles are $\conv\{P_0,P_i,P_{i+1} \}$,
for
$i=1,\ldots,n$.
The area of each triangle
is $1/2$ of the absolute determinant of an appropriate $3\times 3$ matrix.
The formula follows.
\end{proof}
\begin{corollary}\label{cor:lineartime}
Assuming oracle access to $f$ and $f'$,
we can compute
$\vol(\invbreve{U}^*_f(\bm\xi))$ in $\mathcal{O}(n)$ time.
\end{corollary}
\section{Analysis of convex power functions}\label{sec:xtothep}

Convex power functions constitute a broad and flexible class of
increasing convex univariate functions, useful in a wide variety of applications.
Additionally, an ability to handle the power functions $x^k$ for integers $k\geq 2$, already
gives us a lower-bounding method for $f(x):=\exp(x)$ by truncating its Maclaurin series
$\sum_{k=1}^\infty x^k/k!$, and working termwise (on the terms $k\geq 2$).
More generally, we could approach any univariate function $f:\mathbb{R}\rightarrow \mathbb{R}_+$
like this, as long as its Maclaurin series has all
nonnegative coefficients; i.e., when all derivatives at $0$ are nonnegative.
For example, $1/(1-x)^k$ with integer $k\geq 1$ (i.e., the geometric series and its
derivatives), $\sinh(x)$ and $\tan(x)$ for $x<\pi/2$, and $\arcsin(x)$ for $x<1$.
Therefore, analyzing relaxations for power functions,  can have rather broad applicability.

For convenience, let $\invbreve{U}^*_p(\bm\xi)$
denote $\invbreve{U}^*_f(\bm\xi)$,
 with $f(x):=x^p$, $p>1$.

 \subsection{Quadratics}\label{sec:xtothep:quadratics}
We will see that equally-spaced linearization points minimizes the volume of the
relaxation $\invbreve{U}^*_p(\bm\xi)$ when $p=2$.

\begin{theorem}\label{thm:p2n}
Given $n\geq 2$, $0\leq \xi_0:=\ell <\xi_1<\dots<\xi_{n-1}< u=:\xi_n$, we have that
$\xi_i:=\ell + \frac{i}{n}(u-\ell)$, for $i=1,\ldots, n-1$, is the unique minimizer of $\vol(\invbreve{U}^*_2(\bm\xi))$, and the minimum volume is
$\frac{1}{18}(u-\ell)^3+\frac{(u-\ell)^3}{36n^2}$.
\end{theorem}

\begin{proof}
The intersection points $P_i$ are $(\frac{\xi_{i-1}+\xi_i}{2},\xi_{i-1}\xi_i)$.
We have $\tau_i=\frac{\xi_{i-1}+\xi_i}{2}$ for $i=1,\ldots, n+1$,
and
\begin{align*}
\vol(\invbreve{U}^*_2(\bm\xi))&=\frac{1}{6}\sum_{i=1}^n
\adet\begin{pmatrix}
    \tau_0 & \tau_i & \tau_{i+1} \\
    g(\tau_0 ) & g(\tau_i) & g(\tau_{i+1}) \\
    1 & 1  & 1
  \end{pmatrix}\\
  &=\frac{1}{12}\sum_{i=1}^{n}(\xi_{i+1}-\xi_{i-1})(\xi_i-\ell)^2\\
  &=\frac{1}{12}\left[\sum_{i=1}^{n}\xi_i\xi_{i-1}(\xi_{i-1}-\xi_{i})+u^3-2u^2\ell+2u\ell^2-\ell^3\right],
  \end{align*}
and
\begin{align*}
\frac{\partial \vol(\invbreve{U}^*_2(\bm\xi))}{\partial \xi_i} &= \frac{1}{12}(\xi_{i+1}-\xi_{i-1})(2\xi_i-\xi_{i+1}-\xi_{i-1}), ~\text{for}~i=1,\ldots, n-1,\\
\frac{\partial^2 \vol(\invbreve{U}^*_2(\bm\xi))}{\partial \xi_i^2} &= \frac16(\xi_{i+1}-\xi_{i-1}), ~\text{for}~i=1,\ldots, n-1 ,\\
\frac{\partial^2 \vol(\invbreve{U}^*_2(\bm\xi))}{\partial \xi_i\partial \xi_{i+1}} &= \frac16(\xi_i-\xi_{i+1}), ~\text{for}~i=1,\ldots, n-2.
\end{align*}
Therefore, $\nabla^2 \vol(\invbreve{U}^*_2(\bm\xi))$ is a tridiagonal matrix. It is easy to verify that $\nabla^2 \vol(\invbreve{U}^*_2(\bm\xi))$ is  diagonally dominant because $(\xi_{i+1}-\xi_{i-1})= (\xi_{i+1}-\xi_i) +(\xi_{i}-\xi_{i-1})$, thus $\nabla^2 \vol(\invbreve{U}^*_2(\bm\xi))$ is positive semidefinite, i.e., $\vol(\invbreve{U}^*_2(\bm\xi))$ is convex.

The global minimizer satisfies $\nabla\vol(\invbreve{U}^*_2(\bm\xi))=0$, i.e., $2\xi_i-\xi_{i+1}-\xi_{i-1}=0$ for $i=1,\ldots, n-1$. Solving these equations gives us the equally-spaced points. Now a simple calculation
gives the  minimum volume as
\[
\vol(\invbreve{U}^*_2(\bm\xi)) = \frac{1}{12}\left(\frac23(u-\ell)^3+\frac{1}{3n^2}(u-\ell)^3\right)=\frac{1}{18}(u-\ell)^3+\frac{(u-\ell)^3}{36n^2}.
\]
\end{proof}

Letting $n$ go to infinity, we recover the volume of the perspective relaxation for the quadratic $\invbreve{S}^*_2 := \invbreve{S}^*_{f}(\ell,u)$ where $f(x):=x^2$.
\begin{corollary}[\cite{Perspec2019}]\label{cor:S2}
$~\vol(\invbreve{S}^*_2)=\frac{1}{18}(u-\ell)^3$.
\end{corollary}

We can also now easily see that by
using the perspective of our piecewise-linear under-estimator, even with only one (well-placed) non-boundary linearization point, we always outperform the na\"{i}ve  relaxation $\invbreve{S}^0_2 := \invbreve{S}^0_{f}(\ell,u)$, where $f(x):=x^2$.

\begin{corollary}\label{rem:compare}
$\vol(\invbreve{U}^*_2(\bm\xi)) \leq \vol(\invbreve{S}^0_2)$, and with equality only if $n=1$ and $\ell=0$.
\end{corollary}

\begin{proof}
$\vol(\invbreve{S}^0_2)=\frac{1}{18}(u-\ell)^3 + (u^3-\ell^3)/36$ (see \cite{Perspec2019}). Notice that
\[
\frac{(u-\ell)^3}{36n^2}\leq \frac{(u-\ell)^3}{36}\le \frac{u^3-\ell^3}{36}.
\]
The first inequality is strict when $n>1$, and the second is strict when $\ell>0$.
\end{proof}

\subsection{Non-quadratic convex power functions}\label{sec:xtothep:nonquadratics}

Considering $p\not=2$,
even for one non-boundary linearization point,
$\vol(\invbreve{U}^*_p(\bm\xi))$
is not generally convex in $\xi_1$ for $\bm\xi=(\ell,\xi_1,u)$. However, we establish
with Theorem \ref{thm:localmin}
 that any stationary point of $\vol(\invbreve{U}^*_p(\bm\xi))$ is a strict local minimizer.
 Therefore, using any NLP algorithm that can find a stationary point, we are assured that such a
 point is a strict local minimizer.
 Furthermore, we establish with  Theorem \ref{conj:pn}
 that when $1<p\le 2$, we have that $\vol(\invbreve{U}^*_p(\bm\xi))$ is indeed convex in $(\xi_1,\dots,\xi_{n-1})$. Therefore, for $1<p\le 2$,
 using any NLP algorithm that can find a stationary point, we will in fact find a global minimum. For $p>2$, we simplify the gradient condtion $\nabla\vol(\invbreve{U}^*_p(\bm\xi))=0$ and establish with Theorem \ref{thm:unique_stationary} that the volume function has a unique stationary point. We also establish with Theorem \ref{conj:increasing} that the optimal location of each linearization point is increasing in $p$ on $(0,\infty)$. Furthermore, we establish with Theorem \ref{thm:Newtonn} that the iterates of Newton's method have monotonic convergence on this function.
\begin{theorem}\label{thm:localmin}
  For $0\leq \ell < u$, $p>1$, and $\bm\xi:=(\ell,\xi_1,\dots,\xi_{n-1},u)$ $(\ell<\xi_1<\dots<\xi_{n-1}<u)$, if $\bm\xi$ satisfies $\nabla\vol(\invbreve{U}^*_p(\bm\xi))=0$, then $\nabla^2\vol(\invbreve{U}^*_p(\bm\xi))$ is positive definite.
\end{theorem}
\begin{proof}
The intersection points $P_i$ are $\left(\frac{p-1}{p}\frac{\xi_i^p-\xi_{i-1}^p}{\xi_i^{p-1}-\xi_{i-1}^{p-1}},(p-1)\xi_{i-1}^{p-1}\xi_i^{p-1}\frac{\xi_i-\xi_{i-1}}{\xi_i^{p-1}-\xi_{i-1}^{p-1}}\right)$. Let $\tau_{n+1}:=u$, and $\tau_i :=\frac{p-1}{p}\frac{\xi_i^p-\xi_{i-1}^p}{\xi_i^{p-1}-\xi_{i-1}^{p-1}}$ for $i=1,\ldots, n$.
\begin{align*}
&\vol(\invbreve{U}^*_p(\bm\xi))
=\frac{1}{6}\sum_{i=1}^n
\adet\begin{pmatrix}
    \tau_0 & \tau_i & \tau_{i+1} \\
    g(\tau_0 ) & g(\tau_i) & g(\tau_{i+1}) \\
    1 & 1  & 1
  \end{pmatrix}\\
  &~=-\frac{(p-1)^2}{6p}\sum_{i=1}^n\frac{(\xi_i^p-\xi_{i-1}^p)^2}{\xi_i^{p-1}-\xi_{i-1}^{p-1}} +\frac16((p-1)u^{p+1}-u^p\ell+u\ell^p-(p-1)\ell^p)\\
  &~=-\frac{(p-1)^2}{6p}\sum_{i=1}^n\frac{\xi_{i-1}^{p-1}\xi_i^{p-1}(\xi_i-\xi_{i-1})^2}{\xi_i^{p-1}-\xi_{i-1}^{p-1}}+\frac{(p-1)}{6p}(u^{p+1}-\ell^{p+1}) -\frac16(u^p\ell-u\ell^p).
\end{align*}
Therefore, for $i=1,\ldots, n-1$, $\frac{\partial \vol(\invbreve{U}^*_p(\bm\xi))}{\partial \xi_i}=$
\begin{align*}
  & -\frac{(p-1)\xi_i^{p-2}}{6p}\left(\left(\frac{\xi_i^{p}+(p-1)\xi_{i+1}^{p}-p\xi_i\xi_{i+1}^{p-1}}{\xi_{i+1}^{p-1}-\xi_i^{p-1}}\right)^2-\left(\frac{\xi_i^{p}+(p-1)\xi_{i-1}^{p}-p\xi_i\xi_{i-1}^{p-1}}{\xi_{i-1}^{p-1}-\xi_i^{p-1}}\right)^2\right),
\end{align*}
and for $i=1,\ldots, n-2$, $\frac{\partial^2 \vol(\invbreve{U}^*_p(\bm\xi))}{\partial \xi_i\partial \xi_{i+1}}=$
\begin{align*}
&-\frac{(p-1)^2}{3p}\frac{\xi_i^{p-2}\xi_{i+1}^{p-2}[(p-1)\xi_{i+1}^{p}+\xi_i^p-p\xi_{i+1}^{p-1}\xi_i)][\xi_{i+1}^{p}+(p-1)\xi_i^p-p\xi_{i+1}\xi_i^{p-1}]}{(\xi_{i+1}^{p-1}-\xi_{i}^{p-1})^3}.
\end{align*}
For simplicity, we denote for $i=0,1,\dots,n-1$,
$$
b_i:=\frac{(p-1)^2}{3p}\frac{\xi_i^{p-2}\xi_{i+1}^{p-2}[(p-1)\xi_{i+1}^{p}+\xi_i^p-p\xi_{i+1}^{p-1}\xi_i)][\xi_{i+1}^{p}+(p-1)\xi_i^p-p\xi_{i+1}\xi_i^{p-1}]}{(\xi_{i+1}^{p-1}-\xi_{i}^{p-1})^3}.
$$
By Lemma \ref{lem:positive} (See Appendix), we have $b_0\ge0$ and $b_i>0$, for $i=1,2,\ldots, n-1$. Then, for $i=1,2,\ldots, n-1$,
\begin{align*}
  \frac{\partial^2 \vol(\invbreve{U}^*_p(\bm\xi))}{\partial \xi_i^2} & = \frac{p}{\xi_i}\frac{\partial \vol(\invbreve{U}^*_p(\bm\xi))}{\partial \xi_i} +\frac{\xi_{i-1}}{\xi_i}b_{i-1}+\frac{\xi_{i+1}}{\xi_i}b_i~.
\end{align*}

If $\bm\xi$ satisfies $\nabla\vol(\invbreve{U}^*_p(\bm\xi))=0$, then $\nabla^2\vol(\invbreve{U}^*_p(\bm\xi))$ is an $(n-1)\times(n-1)$ symmetric tridiagonal matrix with off-diagonal elements $-b_1,\dots,-b_{n-2}$ and diagonal elements $a_1,\dots,a_{n-1}$ where $a_i := \frac{\xi_{i-1}}{\xi_i}b_{i-1}+\frac{\xi_{i+1}}{\xi_i}b_{i}$.

Notice that $\nabla^2\vol(\invbreve{U}^*_p(\bm\xi))=\lambda e_1e_1^\top + M$, where $\lambda=\frac{\xi_0b_0}{\xi_1}\ge0$, $M :=P D P^\top$, $D := \mathrm{diag}(\frac{\xi_2}{\xi_1}b_1,\allowbreak\frac{\xi_3}{\xi_2}b_2,\dots,\frac{\xi_{n}}{\xi_{n-1}}b_{n-1})$, and $P=[p_{ij}]$ is a lower-triangular matrix with
$$
p_{ij} := \left\{
  \begin{array}{ll}
    1, & i = j~;\\
    -\frac{\xi_{i-1}}{\xi_i} & j = i-1~;\\
    0, & \text{otherwise.}
  \end{array}
\right.
$$
Because $M=PDP^\top$ is positive definite, and $\lambda e_1e_1^\top$ is positive semidefinite, we have that $\nabla^2\vol(\invbreve{U}^*_p(\bm\xi))$ is positive definite.
\end{proof}
\begin{theorem}\label{conj:pn}
For $0\leq \ell < u$, $1<p\leq 2$, and $\bm\xi:=(\ell,\xi_1,\dots,\xi_{n-1},u)$ $(\ell<\xi_1<\dots<\xi_{n-1}<u)$,
$\vol(\invbreve{U}^*_p(\bm\xi))$ is strictly convex in $(\xi_1,\dots,\xi_{n-1})$.
\end{theorem}

\begin{remark}\label{rem:notquasiconvex}
  When $p> 2$, for the single non-boundary linearization point case, we can demonstrate that $\vol(\invbreve{U}^*_p(\bm\xi))$ is quasiconvex in $\xi_1$ (Theorem \ref{thm:concave1}). However, for the multiple non-boundary linearization points case, $\vol(\invbreve{U}^*_p(\bm\xi))$ is no longer guaranteed to be quasiconvex (from computation).
  A necessary condition for the quasiconvexity of $\vol(\invbreve{U}^*_p(\bm\xi))$ is that for all $\bm\xi$ $(\ell<\xi_1<\dots<\xi_{n-1}<u)$, and $d\in\mathbb{R}^{n-1}$, we have
  $$
d^\top \nabla \vol(\invbreve{U}^*_p(\bm\xi))=0 \quad\Rightarrow\quad d^\top \nabla^2 \vol(\invbreve{U}^*_p(\bm\xi)) d\ge 0.
  $$
 (see \cite{boyd2004convex}).
  This is equivalent to: either $\nabla \vol(\invbreve{U}^*_p(\bm\xi))=0$ and $\nabla^2 \vol(\invbreve{U}^*_p(\bm\xi))$ positive semidefinite or $\nabla \vol(\invbreve{U}^*_p(\bm\xi))\ne 0$ and the matrix
  $$
  \begin{bmatrix}
    \nabla^2 \nabla \vol(\invbreve{U}^*_p(\bm\xi)) & \nabla \vol(\invbreve{U}^*_p(\bm\xi))\\
    \nabla \vol(\invbreve{U}^*_p(\bm\xi))^\top & 0
  \end{bmatrix}
  $$
  has exactly one negative eigenvalue. We can easily find examples where this matrix has more than one negative eigenvalue. For example, for $p=3$, $n=3$, $\xi_1=0.2$, $\xi_2=0.8$, the eigenvalues are approximately $-0.03950$, $-0.00086$, and $0.30807$.
\end{remark}

\begin{proof}(Theorem \ref{conj:pn})
Recall that $\nabla^2\vol(\invbreve{U}^*_p(\bm\xi))$ is an $(n-1)\times(n-1)$ symmetric tridiagonal matrix with off-diagonal elements $-b_1,\dots,-b_{n-2}$ and diagonal elements $a_1,\dots,a_{n-1}$ satisfying $a_i = \frac{p}{\xi_i}\frac{\partial \vol(\invbreve{U}^*_p(\bm\xi))}{\partial \xi_i}+ \frac{\xi_{i-1}}{\xi_i}b_{i-1}+\frac{\xi_{i+1}}{\xi_i}b_{i}$, where
$\frac{\partial \vol(\invbreve{U}^*_p(\bm\xi))}{\partial \xi_i}=$
\begin{align*}
  & -\frac{(p-1)\xi_i^{p-2}}{6p}\negthinspace\negthinspace\negthinspace\left(\negthinspace\negthinspace\negthinspace\left(\frac{\xi_i^{p}+(p-1)\xi_{i+1}^{p}-p\xi_i\xi_{i+1}^{p-1}}{\xi_{i+1}^{p-1}-\xi_i^{p-1}}\right)^2
  \negthinspace\negthinspace\negthinspace-\negthinspace\negthinspace
  \left(\frac{\xi_i^{p}+(p-1)\xi_{i-1}^{p}-p\xi_i\xi_{i-1}^{p-1}}{\xi_{i-1}^{p-1}-\xi_i^{p-1}}\right)^2\right)\negthinspace\negthinspace\negthinspace.\\
  &b_i=\frac{(p-1)^2}{3p}\frac{\xi_i^{p-2}\xi_{i+1}^{p-2}[(p-1)\xi_{i+1}^{p}+\xi_i^p-p\xi_{i+1}^{p-1}\xi_i)][\xi_{i+1}^{p}+(p-1)\xi_i^p-p\xi_{i+1}\xi_i^{p-1}]}{(\xi_{i+1}^{p-1}-\xi_{i}^{p-1})^3}.
\end{align*}
To show that $\nabla^2\vol(\invbreve{U}^*_p(\bm\xi))$ is positive definite, we will apply a result from \cite{andjelic2011sufficient} to prove that $a_i>0$ and $\left\{\frac{b_i^2}{a_ia_{i+1}}\right\}_{i=1}^{n-2}$ is a \emph{chain sequence}; that is, there exists a parameter sequence $\{c_i\}_{i=0}^{n-2}$ such that $\frac{b_i^2}{a_ia_{i+1}}=c_i(1-c_{i-1})$ with $0\le c_0<1$ and $0<c_i<1$ for $i\ge1$. Also, we use the fact that if $\{\alpha_i\}$ is a chain sequence, and $0<\beta_i\le \alpha_i$, then $\{\beta_i\}$ is also a chain sequence. Therefore, we only need to show that $a_{i}>0$ and find a parameter sequence $\{c_i\}$ such that $0\le c_0<1$, $0<c_i<1$ for $i\ge 1$, and $0<\frac{b_i^2}{a_ia_{i+1}}\le c_i(1-c_{i-1})$.
Let $c_i :=\frac{d_{i+1}}{a_{i+1}}$,
where
\begin{align*}
  d_i &:= \frac{(p-1)\xi_i^{p-2}}{6\xi_i}\left(\frac{\xi_i^{p}+(p-1)\xi_{i-1}^{p}-p\xi_i\xi_{i-1}^{p-1}}{\xi_{i-1}^{p-1}-\xi_i^{p-1}}\right)^2+\frac{\xi_{i-1}}{\xi_i}b_{i-1}.
\end{align*}
Thus $d_1\ge0$ and $d_i>0$ for $i\ge 2$. Also, letting $t_i:=\frac{\xi_{i}}{\xi_{i+1}}$, $t_0\in[0,1)$, $t_i\in(0,1)$ for $i\ge 1$, we have
\begin{align*}
  &~a_i-d_i = -\frac{(p-1)\xi_i^{p-2}}{6\xi_i}\left(\frac{\xi_i^{p}+(p-1)\xi_{i+1}^{p}-p\xi_i\xi_{i+1}^{p-1}}{\xi_{i+1}^{p-1}-\xi_i^{p-1}}\right)^2+\frac{\xi_{i+1}}{\xi_i}b_{i}\\
  =&~\frac{(p-1)\xi_i^{p-2}}{6\xi_i}\left(\frac{\xi_i^{p}+(p-1)\xi_{i+1}^{p}-p\xi_i\xi_{i+1}^{p-1}}{\xi_{i+1}^{p-1}-\xi_i^{p-1}}\right)\\
  &\times\left(-\frac{\xi_i^{p}+(p-1)\xi_{i+1}^{p}-p\xi_i\xi_{i+1}^{p-1}}{\xi_{i+1}^{p-1}-\xi_i^{p-1}}+\frac{2(p-1)\xi_{i+1}^{p-1}}{p(\xi_{i+1}^{p-1}-\xi_i^{p-1})}\frac{\xi_{i+1}^{p}+(p-1)\xi_{i}^{p}-p\xi_{i+1}\xi_{i}^{p-1}}{\xi_{i+1}^{p-1}-\xi_i^{p-1}}\right)\\
  =&~ \frac{(p-1)\xi_{i+1}^{p-1}t_i^{p-3}(t_i^p+(p-1)-pt_i)}{6p(1-t_i^{p-1})^3}\\
  &~~~~~~\times(p(t_i^{p-1}-1)(t_i^p+(p-1)-pt_i) +2(p-1)((p-1)t_i^p+1-pt_i^{p-1}))\\
  =&~ \frac{(p-1)\xi_{i+1}^{p-1}t_i^{p-3}(t_i^p+(p-1)-pt_i)}{6p(1-t_i^{p-1})^3}\\
  &~~~~~~\times(pt_i(t_i^{p-1}-1)^2 + (p-1)[(p-2)(t_i^p-1)-p(t_i^{p-1}-t_i)]).
\end{align*}
By Lemma \ref{lem:positive} and Lemma \ref{lem:hx}(i) (See Appendix), we have that $a_i-d_i>0$. Therefore, $a_i>d_i\ge0$, and we have constructed $\{c_i\}$ satisfying $0\le c_0<1$ and $0<c_i<1$ for $i\ge 1$. Notice that
\begin{align*}
  d_i &= \frac{(p-1)\xi_{i}^{p-1}((p-1)t_{i-1}^p+1-pt_{i-1}^{p-1})}{6p(1-t_{i-1}^{p-1})^3}\\
  &~~~~~~\times(p(1-t_{i-1}^{p-1})((p-1)t_{i-1}^p+1-pt_{i-1}^{p-1}) +2(p-1)t_{i-1}^{p-1}(t_{i-1}^p+(p-1)-pt_{i-1})),
\end{align*}
$$
  b_i^2 = \frac{(p-1)^4}{9p^2}\frac{\xi_{i+1}^{2(p-1)}t_i^{2(p-2)}((p-1)+t_i^p-p t_i))^2(1+(p-1)t_i^p-pt_i^{p-1})^2}{(1-t_{i}^{p-1})^6}.
$$
We have
\begin{align*}
  &~\frac{a_ia_{i+1}c_i(1-c_{i-1})}{b_i^2} = \frac{d_{i+1}\left(a_i-d_i\right)}{b_i^2}\\
  =&~\frac{1}{4(p-1)^2}\frac{1}{t_i^{p-1}((p-1)t_{i}^p+1-pt_{i}^{p-1})(t_i^p+(p-1)-pt_i)}\\
  &~~~~~~\times(p(1-t_i^{p-1})^2 - (p-1)t_i^{p-1}[(p-2)(t_i^p-1)-p(t_i^{p-1}-t_i)])\\
  &~~~~~~\times(pt_i(t_i^{p-1}-1)^2 + (p-1)[(p-2)(t_i^p-1)-p(t_i^{p-1}-t_i)])\\
  =&~1 +\frac{1}{4(p-1)^2}\frac{p(1-t_i^{p-1})^2}{t_i^{p-1}((p-1)t_{i}^p+1-pt_{i}^{p-1})(t_i^p+(p-1)-pt_i)}\\
  &~~~~~~\times(pt_i(1-t_i^{p-1})^2 - p(p-1)^2t_i^{p-1}(1-t_i)^2+\\
  &~~~~~~~(p-1)(1-t_i^p)[(p-2)(t_i^p-1)-p(t_i^{p-1}-t_i)])\\
  =:&1 +\frac{1}{4(p-1)^2}\frac{p(1-t_i^{p-1})^2}{t_i^{p-1}((p-1)t_{i}^p+1-pt_{i}^{p-1})(t_i^p+(p-1)-pt_i)}W(t_i).
\end{align*}
\begin{align*}
  W'(t) &= -2p(p-1)t^{p-1}[(p-2)(t^p-1)-p(t^{p-1}-t)]\\
  &~~~+p^2[(1-t^{p-1})^2-(p-1)^2t^{p-2}(1-t)^2].
\end{align*}
By Lemma \ref{lem:hx}(i) and Lemma \ref{lem:deltax}(i) (See Appendix), $W'(t)\le 0$ for $t\in(0,1)$. Thus $W(t)\ge W(1)=0$ for $t\in[0,1)$. Therefore, $c_i(1-c_{i-1})\ge \frac{b_i^2}{a_ia_{i+1}}$. We conclude that $\nabla^2\vol(\invbreve{U}^*_p(\bm\xi))$ is positive definite, and $\vol(\invbreve{U}^*_p(\bm\xi))$ is strictly convex.
\end{proof}

\begin{remark}
  Unlike the $p=2$ case (Theorem \ref{thm:p2n}), $\nabla^2\vol(\invbreve{U}^*_p(\bm\xi))$ is not guaranteed to be diagonally dominant. Examples can be easily constructed even for $n=2$; for example, $p=1.5$, $n=2$, $\xi = (0,0.2,0.8,1)$, $\nabla^2\vol(\invbreve{U}^*_p(\bm\xi))\approx\left[\begin{smallmatrix}
    0.1366 &-0.0621\\-0.0621  & 0.0587
  \end{smallmatrix}\right]$.
  This is why we brought in the relatively-sophisticated
  technique of using chain sequences.
\end{remark}
We immediately have the following very-useful result.
\begin{corollary}\label{thm:partial_uniquemin}
For $1<p\le 2$ and fixed $\ell,~u,~n$, $\vol(\invbreve{U}^*_p(\bm\xi))$ has a unique minimizer satisfying $\ell<\xi_1<\dots<\xi_{n-1}<u$.
\end{corollary}

Next we are going to establish that $\vol(\invbreve{U}^*_p(\bm\xi))$ also has a unique minimizer when $p>2$. As mentioned in Remark \ref{rem:notquasiconvex}, $\vol(\invbreve{U}^*_p(\bm\xi))$ is not guaranteed to be quasiconvex when $p>2$. But with some efforts, we are going to show that $\vol(\invbreve{U}^*_p(\bm\xi))$ has a unique stationary point. For $\ell<\xi_1<\dots<\xi_{n-1}<u$, it is easy to see that $\nabla \vol(\invbreve{U}^*_p(\bm\xi))=0$ is equivalent to $F(\bm\xi)=0$, where $F(\bm\xi) = [F_1(\bm\xi),F_2(\bm\xi),\dots,F_{n-1}(\bm\xi)]^\top$,
\begin{equation*}
F_i(\bm\xi) :=-\frac{\xi_i^p+(p-1)\xi_{i+1}^p-p\xi_i \xi_{i+1}^{p-1}}{\xi_{i+1}^{p-1}-\xi_i^{p-1}}+\frac{\xi_i^p+(p-1)\xi_{i-1}^p-p\xi_i  \xi_{i-1}^{p-1}}{\xi_i^{p-1}-\xi_{i-1}^{p-1}}.
\end{equation*}
\begin{lemma}\label{lem:Jacobian} Assume that  $\ell<\xi_1<\dots<\xi_{n-1}<u$. If either: (i) $1<p<2$ and $F(\bm\xi)\ge 0$, or (ii) $p>2$, then $[F'(\bm\xi)]^{-1}$ is nonnegative.
\end{lemma}
\begin{proof}
  $F'(\bm\xi)=\left[\frac{\partial F_i(\bm\xi)}{\partial \xi_j}\right]_{ij}\in\mathbb{R}^{(n-1)\times(n-1)}$, where
\begin{align}
  \frac{\partial F_i(\bm\xi)}{\partial \xi_i} &= \frac{1}{\xi_i}\left(F_i(\bm\xi) -\xi_{i-1}\frac{\partial F_i(\bm\xi)}{\partial \xi_{i-1}} - \xi_{i+1}\frac{\partial F_i(\bm\xi)}{\partial \xi_{i+1}}\right)\label{eqn:partialF1}\\
  &= H(\xi_{i-1},\xi_i) + H(\xi_i,\xi_{i+1})- \frac{\partial F_i(\bm\xi)}{\partial \xi_{i-1}} - \frac{\partial F_i(\bm\xi)}{\partial \xi_{i+1}},\label{eqn:partialF2}\\
  H(y,z) &=\frac{(y^{p-1}-z^{p-1})^2-(p-1)^2y^{p-2}z^{p-2}(y-z)^2}{(y^{p-1}-z^{p-1})^2},\notag\\
  \frac{\partial F_i(\bm\xi)}{\partial \xi_{i-1}}&= -\frac{(p-1)\xi_{i-1}^{p-2}[(p-1)\xi_{i}^p+\xi_{i-1}^p- p\xi_i^{p-1}\xi_{i-1}]}{(\xi_i^{p-1}-\xi_{i-1}^{p-1})^2},\notag\\
  \frac{\partial F_i(\bm\xi)}{\partial \xi_{i+1}}&= -\frac{(p-1)\xi_{i+1}^{p-2}[(p-1)\xi_{i}^p+\xi_{i+1}^p- p\xi_i^{p-1}\xi_{i+1}]}{(\xi_i^{p-1}-\xi_{i+1}^{p-1})^2}.\notag
\end{align}
  First, by Lemma \ref{lem:positive} (See Appendix), we have that all off-diagonal elements of $F'(\bm\xi)$ are nonpositive; thus $F'(\bm\xi)$ is a $Z$-matrix\footnote{A square matrix $A=[a_{ij}]$ (not necessary symmetric) is called a \emph{$Z$-matrix} if all of its off-diagonal entries are nonpositive.}
  $[F'(\bm\xi)]^{-1}\ge 0$ is one of the equivalent conditions that $F'(\bm\xi)$ is an $M$-matrix\footnote{A $Z$-matrix $A$ is  an \emph{$M$-matrix} if it is
  \emph{positive stable}, that is, all of its eigenvalues have positive real parts.
   In fact, the following conditions are equivalent for a $Z$-matrix to be an $M$-matrix: (1) All real eigenvalues of $A$ are positive; (2) $A$ is nonsingular and $A^{-1}$ is nonnegative; (3) $A=LU$ where $L$ is lower triangular and $U$ is upper triangular and all of the diagonal elements of $L,U$ are positive; (4) There exists a vector $x>0$ such that $Ax>0$; see \cite[Theorem 2.5.3]{horn1994topics}.}

  (i) If $1<p<2$ and $F_i(\bm\xi)\ge 0$, then from \eqref{eqn:partialF1} and $\frac{\partial F_1(\bm\xi)}{\partial \xi_0}\le 0$, we have
  $$F'(\bm\xi) = \mathrm{diag}\left(\frac{F_1(\bm\xi)}{\xi_1}, \frac{F_2(\bm\xi)}{\xi_2}, \dots, \frac{F_{n-1}(\bm\xi)}{\xi_{n-1}}\right) - \frac{\xi_0}{\xi_1}\frac{\partial F_1(\bm\xi)}{\partial \xi_0}e_1e_1^\top + LU\ge LU,$$
  where
  \begin{align*}
    L &:= \begin{bmatrix}
      -\frac{\xi_2}{\xi_1}\frac{\partial F_1(\bm\xi)}{\partial \xi_2} & 0 & 0 & \dots & 0\\
      \frac{\partial F_2(\bm\xi)}{\partial \xi_1} & -\frac{\xi_3}{\xi_2}\frac{\partial F_2(\bm\xi)}{\partial \xi_3} & 0 & \dots & 0\\
      0 & \frac{\partial F_3(\bm\xi)}{\partial \xi_2} & -\frac{\xi_4}{\xi_3}\frac{\partial F_3(\bm\xi)}{\partial \xi_4} & \dots & 0\\
      \vdots & \dots & \dots & \ddots & \vdots\\
      0 & \dots & \dots & \frac{\partial F_{n-1}(\bm\xi)}{\partial \xi_{n-2}} & \quad\quad-\frac{\xi_n}{\xi_{n-1}}\frac{\partial F_{n-1}(\bm\xi)}{\partial \xi_n}
    \end{bmatrix},\\
    U &:= \begin{bmatrix}
      1 & -\frac{\xi_1}{\xi_2} & 0 & \dots & 0 \\
      0 & 1 & -\frac{\xi_2}{\xi_3} & \dots & 0 \\
      \vdots & \dots & \vdots & \dots & \vdots\\
      \vdots & \dots & 0 & 1 & -\frac{\xi_{n-2}}{\xi_{n-1}}\\
      0 & \dots & \dots & 0 & 1
    \end{bmatrix}.
  \end{align*}
  All the diagonal elements of $L,U$ are positive, which implies that $LU$ is an $M$-matrix. Thus $F'(\bm\xi)\ge LU$ is also an $M$-matrix\footnote{The result follows from:  if $\hat{x}>0$ and $LU\hat{x}>0$, then $F'(\bm\xi)\hat{x}\ge LU\hat{x}>0$. (See \cite[Theorem 2.5.4]{horn1994topics}.)}.

  (ii) If $p>2$, then by Lemma \ref{lem:deltax}(ii) (See Appendix), $H(y,z)=H(y/z,1)>0$ for any $y\ne z$. Therefore, from \eqref{eqn:partialF2} we have that $F'(\bm\xi)\mathbf{1}>0$ where $\mathbf{1}$ is an all-1 vector, which implies that $F'(\bm\xi)$ is an $M$-matrix.
\end{proof}
\begin{lemma}\label{lem:convexity} Assume that  $\ell<\xi_1<\dots<\xi_{n-1}<u$. (i) If $1<p<2$, then $F_i(\bm\xi)$ is convex; (ii) If $p>2$, then $F_i(\bm\xi)$ is concave.
\end{lemma}
\begin{proof}  We have
  \begin{align*}
    \frac{\partial^2 F_i(\bm\xi)}{\partial \xi_i^2} &=-\frac{\xi_{i-1}}{\xi_i}\frac{\partial^2 F_i(\bm\xi)}{\partial \xi_i\partial\xi_{i-1}} -\frac{\xi_{i+1}}{\xi_i}\frac{\partial^2 F_i(\bm\xi)}{\partial \xi_i\partial\xi_{i+1}},\\
    \frac{\partial^2 F_i(\bm\xi)}{\partial \xi_{i-1}^2} &=-\frac{\xi_{i}}{\xi_{i-1}}\frac{\partial^2 F_i(\bm\xi)}{\partial \xi_i\partial\xi_{i-1}},\\
    \frac{\partial^2 F_i(\bm\xi)}{\partial \xi_{i+1}^2} &= -\frac{\xi_{i}}{\xi_{i+1}}\frac{\partial^2 F_i(\bm\xi)}{\partial \xi_i\partial\xi_{i+1}},
  \end{align*}
  where
  \begin{align*}
    \frac{\partial^2 F_i(\bm\xi)}{\partial \xi_i\partial\xi_{i-1}}&= -\frac{(p-1)^2\xi_{i-1}^{p-2}\xi_i^{p-2}[(2-p)(\xi_i^p-\xi_{i-1}^p) - p(\xi_i\xi_{i-1}^{p-1} -\xi_i^{p-1}\xi_{i-1})]}{(\xi_i^{p-1}-\xi_{i-1}^{p-1})^3},\\
    \frac{\partial^2 F_i(\bm\xi)}{\partial \xi_i\partial\xi_{i+1}}&= -\frac{(p-1)^2\xi_{i+1}^{p-2}\xi_i^{p-2}[(2-p)(\xi_i^p-\xi_{i+1}^p) - p(\xi_i\xi_{i+1}^{p-1} -\xi_i^{p-1}\xi_{i+1})]}{(\xi_i^{p-1}-\xi_{i+1}^{p-1})^3}.
  \end{align*}
  Notice that
  $$
  \nabla^2 F_i(x) =\begin{bmatrix}
  1 & 0 & 0\\
  -\frac{\xi_{i-1}}{\xi_i} & 1 & 0\\
  0 & -\frac{\xi_i}{\xi_{i+1}} & 1
  \end{bmatrix}
  \begin{bmatrix}
    -\frac{\xi_{i}}{\xi_{i-1}}\frac{\partial^2 F_i(\bm\xi)}{\partial \xi_i\partial\xi_{i-1}} & 0 & 0\\
    0 & -\frac{\xi_{i+1}}{\xi_{i}}\frac{\partial^2 F_i(\bm\xi)}{\partial \xi_i\partial\xi_{i+1}} & 0\\
    0 & 0 & 0
  \end{bmatrix}
  \begin{bmatrix}
  1 & -\frac{\xi_{i-1}}{\xi_i} & 0 \\
  0 & 1 & -\frac{\xi_i}{\xi_{i+1}}\\
  0 & 0 & 1
  \end{bmatrix}.
  $$
  By Lemma \ref{lem:hx} (See Appendix), we have that $\frac{\partial^2 F_i(x)}{\partial \xi_i\partial\xi_{i-1}}<0,~\frac{\partial^2 F_i(x)}{\partial \xi_i\partial\xi_{i+1}}<0$ ($>0$) when $1<p<2$ ($p>2$).
  Therefore, $\nabla^2 F_i(x)$ (-$\nabla^2 F_i(x)$) is positive semidefinite if $1<p<2$ ($p>2$), which implies that $F_i(x)$ is convex (concave) when $1<p<2$ ($p>2$).
\end{proof}
\begin{theorem}\label{thm:unique_stationary}
  If $p>2$, there exists a unique $\bm\xi^*$ ($\ell<\xi^*_1<\dots<\xi^*_{n-1}<u$) such that $F(\bm\xi^*)=0$.
\end{theorem}
\begin{proof}
 Suppose that $F(\bm\xi^1)=F(\bm\xi^2)=0$. By Lemma \ref{lem:Jacobian} (See Appendix), we have that $[F'(\bm\xi^1)]^{-1}$ and  $[F'(\bm\xi^2)]^{-1}$ are nonnegative. Also from Lemma \ref{lem:convexity}, we have that $F_i(\bm\xi)$ is concave, which implies that
 \begin{align*}
   &0 = F(\bm\xi^1) - F(\bm\xi^2) \le F'(\bm\xi^2) (\bm\xi^1-\bm\xi^2),\\
   &0 = F(\bm\xi^2) - F(\bm\xi^1) \le F'(\bm\xi^1) (\bm\xi^2-\bm\xi^1).
 \end{align*}
 Therefore,
 \begin{align*}
  &\bm\xi^1 - \bm\xi^2 = [F'(\bm\xi^2)]^{-1} (F'(\bm\xi^2) (\bm\xi^1-\bm\xi^2))\ge 0,\\
  &\bm\xi^2 - \bm\xi^1 = [F'(\bm\xi^1)]^{-1} (F'(\bm\xi^1) (\bm\xi^2-\bm\xi^1))\ge 0,
\end{align*}
which implies $\bm\xi^1=\bm\xi^2$.
\end{proof}
We immediately have the following very-useful result.
\begin{corollary}\label{thm:remaining_uniquemin}
For $p > 2$ and fixed $\ell,~u,~n$, $\vol(\invbreve{U}^*_p(\bm\xi))$ has a unique minimizer satisfying $\ell<\xi_1<\dots<\xi_{n-1}<u$.
\end{corollary}

It is interesting and potentially useful to understand the behavior of the
optimal locations of linearization points as a function of the power $p>1$.
\smallskip

\begin{theorem}\label{conj:increasing}
  For fixed $\ell$ and $u$, and $\ell<\xi_1<\dots<\xi_{n-1}<u$, suppose that $\bm\xi=(\ell,\xi_1,\dots,\allowbreak\xi_{n-1},u)$ minimizes $\vol(\invbreve{U}^*_p(\bm\xi))$. Then $\xi_i$ ($i=1,2,\dots,n-1$) is increasing in $p$ on $(1,\infty)$.
\end{theorem}
\begin{proof}
  By Corollary \ref{thm:partial_uniquemin} and \ref{thm:remaining_uniquemin}, we have that $\bm\xi$ is unique and satisfies $\nabla\vol(\invbreve{U}^*_p(\bm\xi))=0$
  , i.e., $F(\bm\xi)=0$, where
  \begin{equation*}
  F_i(\bm\xi):=-\frac{\xi_i^p+(p-1)\xi_{i+1}^p-p\xi_i \xi_{i+1}^{p-1}}{\xi_{i+1}^{p-1}-\xi_i^{p-1}}+\frac{\xi_i^p+(p-1)\xi_{i-1}^p-p\xi_i  \xi_{i-1}^{p-1}}{\xi_i^{p-1}-\xi_{i-1}^{p-1}} = 0.
  \end{equation*}
  Recall from Lemma \ref{lem:Jacobian} that when $F(\bm\xi)=0$, $[F'(\bm\xi)]^{-1}$ is nonnegative for $p>1$. Let $F_i(p,\bm\xi) := F_i(\bm\xi)$ to emphasize the dependence $p$. By the implicit function theorem, there exists a small neighborhood around $(p,\bm\xi)$ and a function $\bm\Xi(p)$ such that $\bm\Xi(p)=\bm\Xi$, $F(p,\bm\Xi(p)))=0$, and
  $$
  \frac{\partial \bm\Xi(p)}{\partial p} = -\left[\frac{\partial F_i(p,\bm\Xi(p))}{\partial \xi_j}\right]^{-1} \frac{\partial F(p,\bm\Xi(p))}{\partial p}.
  $$
  We claim that $\frac{\partial F(p,\bm\xi)}{\partial p}$ is negative when $F(p,\bm\xi)=0$. Because $[F'(\bm\xi)]^{-1}$ is nonnegative, it follows that $\frac{\partial \bm\Xi(p)}{\partial p}> 0$.

  We only need to prove the above claim.
  \begin{align*}
    \frac{\partial F(p,\bm\xi)}{\partial p} &= \frac{F_i(p,\bm\xi)}{p}\\
    &~~~~-\frac{p(p-1)\xi_{i+1}^{p-1}\xi_i^{p-1}(\xi_{i+1}-\xi_{i})\log \frac{\xi_i}{\xi_{i+1}} +(\xi_{i+1}^p-\xi_i^p)(\xi_{i+1}^{p-1}-\xi_i^{p-1})}{p(\xi_{i+1}^{p-1}-\xi_i^{p-1})^2}\\
    &~~~~-\frac{p(p-1)\xi_{i-1}^{p-1}\xi_i^{p-1}(\xi_{i-1}-\xi_{i})\log \frac{\xi_i}{\xi_{i-1}} +(\xi_{i-1}^p-\xi_i^p)(\xi_{i-1}^{p-1}-\xi_i^{p-1})}{p(\xi_{i-1}^{p-1}-\xi_i^{p-1})^2}.
  \end{align*}
  Then using Lemma \ref{lem:partialp} (See Appendix) and $F_i(p,\bm\xi)=0$, we have
  \begin{align*}
    \frac{\partial F(p,\bm\xi)}{\partial p} < \frac{F_i(p,\bm\xi)}{p} = 0.
  \end{align*}
  \end{proof}

Starting from equally-spaced points, we can numerically compute the minimizer $\bm\xi$ by solving the nonlinear optimality equation $F(\bm\xi) = 0$ via Newton's method (see, e.g. \cite{ortega2000iterative}).  Illustrating Theorem \ref{conj:increasing},
Figure \ref{fig:ximin_n} shows the computed
  $\bm\xi$ for varying $p$, with $n=5$, $\ell=0$, $u=1$.

\begin{figure}[ht]
  \centering
  \includegraphics[scale = 0.4]{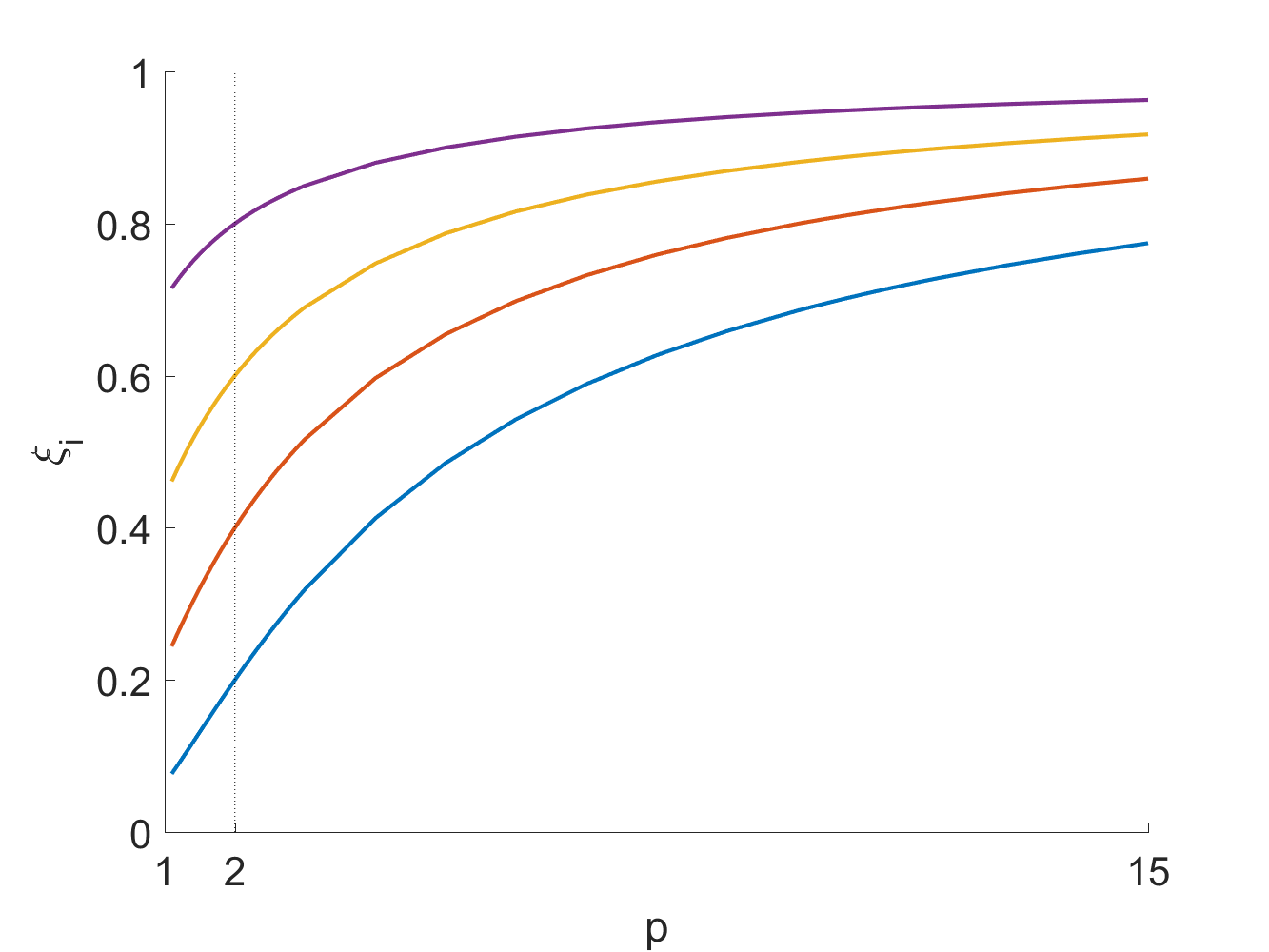}
  \caption{minimizing $\bm\xi$ for varying $p$ ($n=5$, $\ell=0$, $u=1$).}\label{fig:ximin_n}
\end{figure}
In fact, we can show that Newton's method behaves very nicely on this function.
\begin{proposition}\label{prop:initial_F}
  For the equally-spaced linearization points $\xi_i:=\ell+\frac{i}{n}(u-\ell)$, we have $F(\bm\xi)>0$ when $1<p<2$, and $F(\bm\xi)<0$ when $p>2$.
\end{proposition}
\begin{proof}
  We only need to prove the single-linearization-point case, because $\xi_{i}=\frac{\xi_{i-1}+\xi_{i+1}}{2}$ for $i-1,\ldots,n-1$. Let $\hat{\xi}_1(p)$ be the unique optimal solution for power $p$. Then $F(\hat{\xi}_1(p))=0$ and $\hat{\xi}_1(2) = \frac{\ell+u}{2}$ is the equally-spaced linearization point. By Lemma \ref{lem:Jacobian}, we have that $F'(\hat{\xi}_1(p))>0$.

  For $1<p<2$, by Theorem \ref{conj:increasing}, $\hat{\xi}_1(p)\le \hat{\xi}_1(2)$. Therefore,
  $$F(\hat{\xi}_1(2))\ge F(\hat{\xi}_1(p)) + F'(\hat{\xi}_1(p))(\hat{\xi}_1(2)-\hat{\xi}_1(p))\ge 0,$$
  because of the convexity of $F(\xi_1)$ (Lemma \ref{lem:convexity}(i)).

  For $p>2$, by Theorem \ref{conj:increasing}, $\hat{\xi}_1(p)\ge \hat{\xi}_1(2)$. Therefore,
  $$F(\hat{\xi}_1(2))\le F(\hat{\xi}_1(p)) + F'(\hat{\xi}_1(p))(\hat{\xi}_1(2)-\hat{\xi}_1(p))\le 0,$$
  because of the concavity of $F(\xi_1)$ (Lemma \ref{lem:convexity}(ii)).
\end{proof}
\begin{theorem}\label{thm:Newtonn}
  Starting from an initial point $x^0=(\ell+\frac{(u-\ell)}{n},\dots,\ell+\frac{i(u-\ell)}{n},\dots,\ell+\frac{(n-1)(u-\ell)}{n})^\top$, construct the Newton's-method sequence $\{x^k\}$ by iterating
  $$
  x^{k+1} := x^k - [F'(x^k)]^{-1} F(x^k).
  $$
  Then $\{x^k\}$ is monotonically decreasing (increasing) to $x^*$ when $1<p<2$ (respectively, $p>2$), where $x^*$ satisfies $F(x^*)=0$.
\end{theorem}
\begin{proof}
  The result follows from Lemma \ref{lem:Jacobian}, Lemma \ref{lem:convexity} and the ``Monotone Newton Theorem'' \cite[Theorem 13.3.4]{ortega2000iterative}. In the Appendix, we provide a short direct proof.
\end{proof}
\begin{remark}
  For the case of a single non-boundary linearization point, the result also directly follows from the facts that $F'(\xi_1)\ne 0$ and $F(\xi_1)F''(\xi_1)>0$ for all $\xi_1$ between $x^0$ and $\hat{\xi}_1(p)$ (See \cite{mott1957newton}).
\end{remark}

\subsection{Optimal placement of a single non-boundary linearization point}\label{sec:xtothep:single}

It is interesting to make
a detailed study of optimal placement of
 a single non-boundary linearization point, as it relates to
 necessary optimality
conditions for $\bm\xi$, and it can give us a means to carry out a fast parallel
coordinate-descent style algorithm. In this direction, we will establish that
 $\vol(\invbreve{U}^*_p(\ell,\xi_1,u))$
has a unique minimizer.
\begin{theorem}\label{thm:concave1}~
\begin{itemize}
\item[(i)] If $1<p\le 2$, then $\vol(\invbreve{U}^*_p(\ell,\xi_1,u))$ is strictly convex in $\xi_1$.
\item[(ii)] If $p>2$, then $\vol(\invbreve{U}^*_p(\ell,\xi_1,u))$ is quasiconvex in $\xi_1$.
\end{itemize}
\end{theorem}

\begin{proof}
(i) follows directly from Theorem \ref{conj:pn}. (ii) follows directly from Theorem \ref{thm:localmin} (when $\frac{d}{d\xi_1}\vol(\invbreve{U}^*_p(\ell,\xi_1,u))=0$, $\frac{d^2}{d\xi_1^2}\vol(\invbreve{U}^*_p(\ell,\xi_1,u))>0$).
\end{proof}

We immediately have the following very-useful result.
\begin{corollary}\label{thm:uniquemin1}
For all $p>1$, $\vol(\invbreve{U}^*_p(\ell,\xi_1,u))$ has a unique minimizer on $(\ell,u)$.
\end{corollary}

\begin{proposition}\label{rem:nearell}
For all $p>2$, $\vol(\invbreve{U}^*_p(\ell,\xi_1,u))$ is convex in  $\xi_1$
 to the right of the minimizer,
 and not convex  near $\ell$.
\end{proposition}

\begin{proof}
$$
\frac{\partial^2 \vol(\invbreve{U}^*_p(\ell,\xi_1,u))}{\partial \xi_1^2} = \frac{p}{\xi_1}\frac{\partial \vol(\invbreve{U}^*_p(\ell,\xi_1,u))}{\partial \xi_1}+\frac{\ell}{\xi_1}b_\ell + \frac{u}{\xi_1}b_u,
$$
where
\begin{align*}
  &~~\frac{p}{\xi_1}\frac{\partial \vol(\invbreve{U}^*_p(\ell,\xi_1,u))}{\partial \xi_1}\\
  &=-\frac{(p-1)\xi_1^{p-2}}{6p}\left(\left(\frac{\xi_1^{p}+(p-1)u^{p}-p\xi_1u^{p-1}}{u^{p-1}-\xi_1^{p-1}}\right)^2-\left(\frac{\xi_1^{p}+(p-1)\ell^{p}-p\xi_1\ell^{p-1}}{\xi_{1}^{p-1}-\ell^{p-1}}\right)^2\right),\\
  &b_\ell =\frac{(p-1)^2}{3p}\frac{\xi_1^{p-2}\ell^{p-2}[(p-1)\xi_{1}^{p}+\ell^p-p\xi_{1}^{p-1}\ell)][\xi_{1}^{p}+(p-1)\ell^p-p\xi_{i}\ell^{p-1}]}{(\xi_{1}^{p-1}-\ell^{p-1})^3}>0,\\
  &b_u = \frac{(p-1)^2}{3p}\frac{\xi_1^{p-2}u^{p-2}[(p-1)\xi_{1}^{p}+u^p-p\xi_{1}^{p-1}u)][\xi_{1}^{p}+(p-1)u^p-p\xi_{1}u^{p-1}]}{(u^{p-1}-\xi_{1}^{p-1})^3}>0.
\end{align*}
Suppose that $\xi_1^*$ is the minimizer of $\vol(\invbreve{U}^*_p(\ell,\xi_1,u))$.  By Theorem \ref{thm:logconcave}, we have that
$$\frac{\partial \log h_p(\xi_1)}{\partial \xi_1}=-\frac{1}{h_p(\xi_1)}\frac{\partial \vol(\invbreve{U}^*_p(\ell,\xi_1,u))}{\partial \xi_1}$$
is decreasing on $(\ell,u)$. Therefore,
\[
\frac{\partial \vol(\invbreve{U}^*_p(\ell,\xi_1,u))}{\partial \xi_1}
\left\{
  \begin{array}{ll}
    <0, & \hbox{for $\xi_1\in(\ell,\xi_1^*)$ ;} \\
    >0, & \hbox{for $\xi_1\in(\xi_1^*,u)$.}
  \end{array}
\right.
\]
%
%
So $\frac{\partial^2 \vol(\invbreve{U}^*_p(\ell,\xi_1,u))}{\partial \xi_1^2}>0$ for $\xi_1\in(\xi_1^*,u)$.

Next, we demonstrate that $\frac{\partial^2 \vol(\invbreve{U}^*_p(\ell,\xi_1,u))}{\partial \xi_1^2}$ can be negative near $\ell$ when $\ell/u$ is small enough. Notice that
$\lim\limits_{\xi_1\rightarrow\ell}b_\ell=0$, and
$$
\lim\limits_{\xi_1\rightarrow\ell}\frac{\partial \vol(\invbreve{U}^*_p(\ell,\xi_1,u))}{\partial \xi_1} = -\frac{(p-1)\ell^{p-2}}{6p}\left(\frac{\ell^p+(p-1)u^p-p\ell u^{p-1}}{u^{p-1}-\ell^{p-1}}\right)^2.
$$
Therefore,
\begin{align*}
&\lim\limits_{\xi_1\rightarrow\ell}\frac{\partial^2 \vol(\invbreve{U}^*_p(\ell,\xi_1,u))}{\partial \xi_1^2} \\
=~& -\frac{(p-1)\ell^{p-3}}{6}\left(\frac{\ell^p+(p-1)u^p-p\ell u^{p-1}}{u^{p-1}-\ell^{p-1}}\right)^2+\frac{u}{\ell}\lim\limits_{\xi_1\rightarrow\ell}b_u\\
=~&-\frac{(p-1)\ell^{p-3}[(p-1)u^p+\ell^p-pu^{p-1}\ell]}{6p(u^{p-1}-\ell^{p-1})^3}\\
&~\times\Big[(p-1)u^{p-1}[(p-2)(u^p-\ell^p)-p u \ell(u^{p-2}-\ell^{p-2})]-p\ell(u^{p-1}-\ell^{p-1})^2\Big]\\
:=~&-\frac{(p-1)\ell^{p-3}[(p-1)u^p+\ell^p-pu^{p-1}\ell]}{6p(u^{p-1}-\ell^{p-1})^3}u^{2p-1}k_1(\frac{\ell}{u}),
\end{align*}
where $k_1(t) := (p-1)[(p-2)(1-t^p)-p(t-t^{p-1})]-pt(1-t^{p-1})^2$.
Notice that $\lim\limits_{t\rightarrow 0}k_1(t)=(p-1)(p-2)>0$, because $p>2$. Thus, when $\ell/u$ tends to $0$, $\frac{\partial^2 \vol(\invbreve{U}^*_p(\ell,\xi_1,u))}{\partial \xi_1^2}$ is negative.
\end{proof}

Even though $\vol(\invbreve{U}^*_p(\ell,\xi_1,u))$ is not generally convex in $\xi_1$ for $p>2$,
through a simple transformation, we can finds its unique minimizer (which we already know exists because it is quasiconvex) by equivalently maximizing a
related strictly concave function.
\begin{theorem}\label{thm:logconcave}
If $p>2$, then $h_p(\xi_1):=C-\vol(\invbreve{U}^*_p(\ell,\xi_1,u))$ is strictly log-concave, where
$$
C = \frac{((p-1)u^p+\ell^p-pu^{p-1}\ell)(u^p+(p-1)\ell^p-pu\ell^{p-1})}{6p(u^{p-1}-\ell^{p-1})}.
$$
\end{theorem}
\begin{proof}
$h_p(\xi_1) = \frac{(p-1)^2(u^{p-1}-\ell^{p-1})}{6p}q_1(\xi_1)q_2(\xi_1)$, where $q_1(x) = \left(\frac{u^p-\ell^p}{u^{p-1}-\ell^{p-1}}-\frac{x^p-\ell^p}{x^{p-1}-\ell^{p-1}}\right)$,
 and $q_2(x) = \left(\frac{x^p-u^p}{x^{p-1}-u^{p-1}}-\frac{u^p-\ell^p}{u^{p-1}-\ell^{p-1}}\right)$.
We calculate
\begin{align*}
q_1'(x)&=-\frac{x^{p-2}[x^p-\ell^p-p\ell^{p-1}(x-\ell)]}{(x^{p-1}-\ell^{p-1})^2};\\
q_1''(x)&=-\frac{(p-1)x^{p-3}\ell^{p-1}[(p-2)(x^p-\ell^p)-p\ell x(x^{p-2}-\ell^{p-2})]}{(x^{p-1}-\ell^{p-1})^3}.
\end{align*}
Similarly,
\begin{align*}
q_2'(x)&=\frac{x^{p-2}[x^p-u^p-p u^{p-1}(x-u)]}{(x^{p-1}-u^{p-1})^2};\\
q_2''(x)&=\frac{(p-1)x^{p-3}u^{p-1}[(p-2)(x^p-u^p)-p u x(x^{p-2}-u^{p-2})]}{(x^{p-1}-u^{p-1})^3}.
\end{align*}
Because of Lemma \ref{lem:positive}(ii) (See Appendix), $q_1'(x)< 0$, $q_2'(x)>0$ on $(\ell,u)$. Thus $q_1(x)>q_1(u)=0$ and $q_2(x)>q_2(\ell)=0$. Because of Lemma \ref{lem:hx}(ii) (See Appendix), $q_1''(x)<0$, $q_2''(x)>0$.

We are going to show that $q_1(x)$ and $q_2(x)$ is strictly log-concave for $p> 2$.
$$(\log q_1(x))'' = \frac{q_1(x)q_1''(x) - (q_1'(x))^2}{q_1(x)^2}< 0.$$
Note that $q_2''(x)>0$ and $q_2(x)\le \frac{x^p-u^p}{x^{p-1}-u^{p-1}}-u=\frac{x^{p-1}(x-u)}{x^{p-1}-u^{p-1}}$, thus
\begin{align*}
&q_2(x)q_2''(x)-(q_2'(x))^2 \\
\le~& \frac{x^{p-1}(x-u)}{x^{p-1}-u^{p-1}}\frac{(p-1)x^{p-3}u^{p-1}[(p-2)(x^p-u^p)-p u x(x^{p-2}-u^{p-2})]}{(x^{p-1}-u^{p-1})^3}\\
&-\frac{x^{2(p-2)}[x^p-u^p-p u^{p-1}(x-u)]^2}{(x^{p-1}-u^{p-1})^4}\\
=~& \frac{x^{2(p-2)}}{(x^{p-1}-u^{p-1})^4}\Big[(p-1)u^{p-1}(x-u)[(p-2)(x^p-u^p)-p u x(x^{p-2}-u^{p-2})]\\
&\phantom{blankblankblank}-[x^p-u^p-p u^{p-1}(x-u)]^2\Big]\\
=& \frac{x^{2(p-2)}}{(x^{p-1}-u^{p-1})^4} \Big[-(p-1)u^{p-2}(x-u)^2[u^p-x^p-px^{p-1}(u-x)] \\
&\phantom{blankblankblank}-x^2[(x^{p-1}-u^{p-1})^2 - (p-1)^2u^{p-2}x^{p-2
}(x-u)^2]\Big].
\end{align*}
By Lemma \ref{lem:positive}(ii) and Lemma \ref{lem:deltax}(ii) (See Appendix), we have $(\log q_2(x))''<0$.

Therefore, $h_p(x)=\frac{(p-1)^2(u^{p-1}-\ell^{p-1})}{6p}q_1(x)q_2(x)$ is the product of two strictly log-concave function and is thus strictly log-concave.
\end{proof}

Next, we provide some bounds on the minimizing $\xi_1$.
This can be useful for determining a reasonable initial point for
a minimization algorithm (better than equally spaced) or even for a reasonable static
rule for selecting linearization points. Additionally,
we can see these bounds as necessary conditions for a
minimizer.

\begin{theorem}\label{conj:location}
For fixed $\ell$ and $u$, assume that $\xi_1$ minimizes $\vol(\invbreve{U}^*_p(\ell,\xi_1,u))$, then
\begin{itemize}
\item[(i)] if $p=2$, then $\xi_1=\frac{u+\ell}{2}$;
\item[(ii)] if $1<p< 2$, then
$$
\left(\frac{u^{p-1}+\ell^{p-1}}{2}\right)^{\frac{1}{p-1}}<\frac{(p-1)(u^p-\ell^p)}{p(u^{p-1}-\ell^{p-1})}<\xi_1<\left(\frac{u^p-\ell^p}{p(u-\ell)}\right)^{\frac{1}{p-1}}<\frac{u+\ell}{2};
$$
\item[(iii)] if $p>2$, then
$$
\left(\frac{u^{p-1}+\ell^{p-1}}{2}\right)^{\frac{1}{p-1}}>\frac{(p-1)(u^p-\ell^p)}{p(u^{p-1}-\ell^{p-1})}>\xi_1>\left(\frac{u^p-\ell^p}{p(u-\ell)}\right)^{\frac{1}{p-1}}>\frac{u+\ell}{2}.
$$
 \end{itemize}
\end{theorem}

\begin{proof}
(i) follows directly from Theorem \ref{thm:p2n} when $n=2$. We only prove (ii), because (iii) follows a similar proof. $\xi_1$ satisfies the optimal condition $\frac{d}{d\xi_1}\vol(\invbreve{U}^*_p(\ell,\xi_1,u))=0$, which is equivalent to
\begin{equation*}
F(x):=\frac{x^p+(p-1)\ell^p-p\ell^{p-1}x}{x^{p-1}-\ell^{p-1}}-\frac{x^p+(p-1)u^p-p u^{p-1}x}{u^{p-1}-x^{p-1}} = 0.
\end{equation*}
First, note that if $1<p<2$, and $x_0$ satisfies $F(x_0)<0$, then $\xi_1>x_0$; if $x_0$ satisfies $F(x_0)>0$, then $\xi_1<x_0$.

For the lower bound, notice that
\begin{align*}
  F(x) &=\frac{x^p+(p-1)\ell^p-p\ell^{p-1}x}{x^{p-1}-\ell^{p-1}}-\frac{x^p+(p-1)u^p-p u^{p-1}x}{u^{p-1}-x^{p-1}}\\
  &=-(x^p+(p-1)u^p-p u^{p-1}x)\left(\frac{1}{u^{p-1}-x^{p-1}}-\frac{1}{x^{p-1}-\ell^{p-1}}\right)\\
  &~~~~~~ -\frac{(p-1)(u^p-\ell^p)-p(u^{p-1}-\ell^{p-1})x}{x^{p-1}-\ell^{p-1}}.
\end{align*}
Let $\underline{\xi}_1 := \frac{(p-1)(u^p-\ell^p)}{p(u^{p-1}-\ell^{p-1})}$. To show $F(\underline{\xi}_1)<0$, we only need to show that $\underline{\xi}_1^{p-1}-\ell^{p-1}>u^{p-1}-\underline{\xi}_1^{p-1}$, i.e., $\underline{\xi}_1>\left(\frac{u^{p-1}+\ell^{p-1}}{2}\right)^{\frac{1}{p-1}}$, which is the first inequality. Then we could conclude that $\xi_1>\underline{\xi}_1$.

To show the first inequality, we take logarithm on both sides and let $t:=\frac{\ell}{u}$. Then the inequality that we are going to prove is
$$
J(t):=\log(1-t^p)-\log(1-t^{p-1}) +\log\frac{p-1}{p} - \frac{1}{p-1}(\log (t^{p-1}+1) -\log{2})>0
$$
Notice that $\lim\limits_{t\rightarrow 1^-}J(t)=\log\frac{p}{p-1} +\log\frac{p-1}{p}=0$, and
\begin{align*}
J'(t) &= \frac{pt^{p-1}}{t^p-1} - \frac{(p-1)t^{p-2}}{t^{p-1}-1} -\frac{1}{p-1}\frac{(p-1)t^{p-2}}{t^{p-1}+1}\\
&=\frac{t^{p-2}((p-2)(1-t^p)-pt(1-t^{p-2}))}{(t^p-1)(t^{p-1}-1)(t^{p-1}+1)}.
\end{align*}
By Lemma \ref{lem:hx}(i) (See Appendix), $J'(t)<0$ on $(0,1)$. Thus $J(t)>0$ for $t\in(0,1)$.

For the upper bound, first we claim that for $0<t<1$, we have
\begin{equation}\label{eqn:2-p/3}
  \frac{t^p+(p-1)-pt}{1+(p-1)t^p-p t^{p-1}} > t^{\frac{2-p}{3}}.
\end{equation}
To prove the claim, let $K(t):=t^{\frac{2-p}{3}}(1+(p-1)t^p-p t^{p-1})-t^p-(p-1)+pt$.
\begin{align*}
  K'(t) &= \frac{2-p}{3}t^{-\frac{p+1}{3}} + (p-1)\frac{2(p+1)}{3}t^{\frac{2p-1}{3}} -p\frac{2p-1}{3}t^{\frac{2p-4}{3}} -pt^{p-1} +p\\
  &=t^{-\frac{p+1}{3}}\left(\frac{2-p}{3}+\frac{2(p-1)(p+1)}{3}t^{p} - \frac{p(2p-1)}{3}t^{p-1} - pt^{\frac{4p-2}{3}}+ pt^{\frac{p+1}{3}}\right)\\
  &=:t^{-\frac{p+1}{3}} K_1(t).\\
  K_1'(t)&=\frac{d }{dt}\left(t^{\frac{p+1}{3}}K'(t)\right) \\
  =&\frac{2p(p^2-1)}{3}t^{p-1} - \frac{p(p-1)(2p-1)}{3}t^{p-2} - \frac{p(4p-2)}{3}t^{\frac{4p-5}{3}}+ \frac{p(p+1)}{3}t^{\frac{p-2}{3}}\\
  =& \frac{p}{3}t^{\frac{p-2}{3}}\left(2(p^2-1)t^{\frac{2p-1}{3}} - (2p-1)(p-1)t^{\frac{2p-4}{3}} - 2(2p-1)t^{p-1}+ (p+1)\right)\\
  =:&\frac{p}{3}t^{\frac{p-2}{3}}K_2(t).\\
  K_2'(t)&=\frac{d}{dt}\left(\frac{3}{p}t^{\frac{2-p}{3}}\frac{d }{dt}\left(t^{\frac{p+1}{3}}K'(t)\right)\right)  \\
  =& 2(2p-1)(p-1)t^{\frac{2p-7}{3}}\left( \frac{p+1}{3}t - \frac{p-2}{3} - t^{\frac{p+1}{3}}\right)\\
  =& -2(2p-1)(p-1)t^{\frac{2p-7}{3}}\left((t^{\frac{p+1}{3}}-1)- \frac{p+1}{3}(t-1)\right) > 0.\\
\end{align*}
The last inequality follows from the strict concavity of function $x^{\frac{p+1}{3}}$ when $1<p<2$. Because $K_2(1)=0$, we have $K_2(t)< 0$ on $(0,1)$, which implies $K_1(t)$ is decreasing on $(0,1)$. Along with $K_1(1)=0$, which implies $K_1(t)> 0$ on $(0,1)$. Therefore, $K(t)$ is increasing on $(0,1)$, and $K(t)< K(1)=0$, which proves the claim.
Letting $\overline{\xi}_1:=\left(\frac{u^p-\ell^p}{p(u-\ell)}\right)^{\frac{1}{p-1}}$, and $t:=\frac{\ell}{u}$, we have
\begin{align*}
&\ell^p -p\overline{\xi}_1^{p-1}\ell = u^p - p\overline{\xi}_1^{p-1}u,\\
&\frac{\overline{\xi}_1^{p-1}-\ell^{p-1}}{u^{p-1}-\overline{\xi}_1^{p-1}} = \frac{p(u-\ell)(\overline{\xi}_1^{p-1}-\ell^{p-1})}{p(u-\ell)(u^{p-1}-\overline{\xi}_1^{p-1})} = \frac{(p-1)t^p+1-pt^{p-1}}{t^p+(p-1)-pt}.
\end{align*}
We are going to show that $F(\overline{\xi}_1)>0$. Letting $h(x):=\frac{x^p+(p-1)-px}{x(x^{p-1}-1)}$, we have
$$h'(x)=\frac{(p-1)((p-1)x^p+1-px^{p-1})}{x^2(x^{p-1}-1)^2},$$
and
\begin{align*}
  H(t):=\frac{F(\overline{\xi}_1)}{\overline{\xi}_1} &=h\left(\frac{\overline{\xi}_1}{\ell}\right) + h\left(\frac{\overline{\xi}_1}{u}\right)\\
  &=h\left(\left(\frac{(t^p-1)}{pt^{p-1}(t-1)}\right)^{\frac{1}{p-1}}\right) + h\left(\left(\frac{(t^p-1)}{p(t-1)}\right)^{\frac{1}{p-1}}\right).
\end{align*}
\begin{align*}
  \frac{d H(t)}{dt} &=-h'\left(\frac{\overline{\xi}_1}{\ell}\right)\frac{1}{p-1}\left(\frac{\overline{\xi}_1}{\ell}\right)^{2-p}\frac{t^p+(p-1)-pt}{pt^p(t-1)^2}\\
  &~~~~~ + h'\left(\frac{\overline{\xi}_1}{u}\right)\frac{1}{p-1}\left(\frac{\overline{\xi}_1}{u}\right)^{2-p}\frac{(p-1)t^p+1-pt^{p-1}}{p(t-1)^2}\\
  &=-\frac{\ell^{p-2}}{p(t-1)^2}\left(\frac{\overline{\xi}_1}{u}\right)^{-p} \frac{((p-1)\overline{\xi}_1^p+\ell^p-p\overline{\xi}_1^{p-1}\ell)(t^p+(p-1)-pt)}{(\overline{\xi}_1^{p-1}-\ell^{p-1})^2}\\
  &~~~~~ +\frac{u^{p-2}}{p(t-1)^2}\left(\frac{\overline{\xi}_1}{u}\right)^{-p} \frac{((p-1)\overline{\xi}_1^p+u^p-p\overline{\xi}_1^{p-1}u)((p-1)t^p+1-pt^{p-1})}{(u^{p-1}-\overline{\xi}_1^{p-1})^2}\\
  &=\frac{\ell^{p-2}}{p(t-1)^2}\left(\frac{\overline{\xi}_1}{u}\right)^{-p} ((p-1)\overline{\xi}_1^p+\ell^p-p\overline{\xi}_1^{p-1}\ell)\\
  &~~~~~ \times \left(-\frac{(t^p+(p-1)-pt)}{(\overline{\xi}_1^{p-1}-\ell^{p-1})^2}+\frac{t^{2-p}((p-1)t^p+1-pt^{p-1})}{(u^{p-1}-\overline{\xi}_1^{p-1})^2}\right)\\
  &=\frac{\ell^{p-2}}{p(t-1)^2}\left(\frac{\overline{\xi}_1}{u}\right)^{-p} ((p-1)\overline{\xi}_1^p+\ell^p-p\overline{\xi}_1^{p-1}\ell)\\
  &~~~~~ \times \frac{(p-1)t^p+1-pt^{p-1}}{(u^{p-1}-\overline{\xi}_1^{p-1})^2} \left(t^{2-p}-\left(\frac{t^p+(p-1)-pt}{(p-1)t^p+1-pt^{p-1}}\right)^3\right) < 0.\\
\end{align*}
The last inequality follows from \eqref{eqn:2-p/3}. Therefore, along with $\lim\limits_{t\rightarrow 1^-}H(t)=0$, we have $H(t)> 0$ for $t\in (0,1)$, which implies $F(\overline{\xi}_1)>0$ and $\xi_1<\overline{\xi}_1$.

To show that
$
\frac{u+\ell}{2}>\overline{\xi}_1,
$
we take logarithm on both sides and let $t:=\frac{\ell}{u}$. Then the inequality that we are going to prove is
$$
L(t):=\log(1-t^p)-\log(1-t) -\log p - (p-1)(\log (t+1) -\log{2})<0.
$$
Notice that $\lim\limits_{t\rightarrow 1^-}L(t)=0$, and
\begin{align*}
L'(t) &= \frac{pt^{p-1}}{t^p-1} - \frac{1}{t-1} -\frac{p-1}{t+1}\\
&=\frac{(p-2)(t^p-1)-pt(t^{p-2}-1)}{(t^p-1)(t^2-1)}.
\end{align*}
By Lemma \ref{lem:hx}(i) (See Appendix), $L'(t)>0$ on $(0,1)$. Thus $L(t)<0$ for $t\in(1,\infty)$.
\end{proof}

Just as we determined the optimal location of a linearization
point as $p$ varies  (Theorem \ref{conj:increasing}),
we now determine the behavior of these bounds (Theorem \ref{conj:location})
when $p$ varies. Toward this goal, let $t:=\frac{\ell}{u}$, and let
$$
\Delta(p,t):= \frac{\overline{\xi}_1-\underline{\xi}_1}{u-\ell} = \frac{1}{1-t}\left(\left(\frac{1-t^p}{p(1-t)}\right)^{\frac{1}{p-1}}-\frac{(p-1)(1-t^p)}{p(1-t^{p-1})}\right),
$$
where
 $\underline{\xi}_1 := \frac{(p-1)(u^p-\ell^p)}{p(u^{p-1}-\ell^{p-1})}$ and
$\overline{\xi}_1:=\left(\frac{u^p-\ell^p}{p(u-\ell)}\right)^{\frac{1}{p-1}}$.
We will demonstrate that the behavior of  $\Delta(p,t)$  can be bounded, in a useful way,
by the behavior of  $\Delta(p,0)$. Then we will analyze $\Delta(p,0)$.

\begin{theorem}\label{prop:deltamono}~%
  \begin{itemize}
    \item[(i)] For $1<p<2$, $\Delta(p,t)$ is decreasing in $t$, implying that $0<\Delta(p,t)\le\Delta(p,0)$;
    \item[(ii)] for $p>2$, $(1-t)\Delta(p,t)$ is increasing in $t$, implying that $0>(1-t)\Delta(p,t)\ge\Delta(p,0)$.
  \end{itemize}
\end{theorem}
\begin{proof}
  (i) We will demonstrate that the derivative of $\Delta(p,t)$ is negative when $1<p<2$.
\begin{align*}
  \frac{\partial \Delta(p,t)}{\partial t} &= \frac{1}{(1-t)^2}\left[\frac{\overline{\xi}_1}{\underline{\xi}_1} - \left(\frac{(p-1)^2t^{p-2}(1-t)^2}{p(1-t^{p-1})^2}+\frac{p-1}{p}\right)\right].
\end{align*}
Let $\chi(t):= \log\left(\frac{\overline{\xi}_1}{\underline{\xi}_1}\right) - \log\left(\frac{(p-1)^2t^{p-2}(1-t)^2}{p(1-t^{p-1})^2}+\frac{p-1}{p}\right)$. Then
\begin{align*}
  \frac{\partial \chi(t)}{\partial t}&=\frac{(1-t^{p-1})^2-(p-1)^2t^{p-2}(1-t)^2}{(p-1)(1-t)(1-t^{p-1})(1-t^p)} - \frac{(p-1)t^{p-3}(1-t)[(p-2)(1-t^p)-p(t-t^{p-1})]}{(1-t^{p-1})[(1-t^{p-1})^2+(p-1)t^{p-2}(1-t)^2]}.
\end{align*}
We claim that
$$
0>\frac{(1-t^{p-1})^2-(p-1)^2t^{p-2}(1-t)^2}{(p-1)(1-t)}>\frac{(p-2)(1-t^p)-p(t-t^{p-1})}{2t}.
$$
Then
\begin{align*}
  \frac{\partial \chi(t)}{\partial t}&>\frac{(p-2)(1-t^p)-p(t-t^{p-1})}{t(1-t^{p-1})(1-t^p)}\left(\frac12 - \frac{(p-1)t^{p-2}(1-t)(1-t^p)}{(1-t^{p-1})^2+(p-1)t^{p-2}(1-t)^2}\right)\\
  &=\frac{(p-2)(1-t^p)-p(t-t^{p-1})}{t(1-t^{p-1})(1-t^p)}\left(\frac12 - \frac{(p-1)t^{p-2}(1-t)(1-t^p)}{(1-t^{p-1})^2+(p-1)t^{p-2}(1-t)^2}\right).
\end{align*}
Notice that $(p-2)(1-t^p)-p(t-t^{p-1})<0$, and
\begin{align*}
  \frac{(p-1)t^{p-2}(1-t)(1-t^p)}{(1-t^{p-1})^2+(p-1)t^{p-2}(1-t)^2}&>\frac{(p-1)t^{p-2}(1-t)(1-t^p)}{(p-1)^2t^{p-2}(1-t)^2+(p-1)t^{p-2}(1-t)^2}\\
  &=\frac{1-t^p}{p(1-t)}>\frac{1}{p}>\frac12.
\end{align*}
Therefore, $\frac{\partial \chi(t)}{\partial t}>0$, and hence $\chi(t)<\lim_{t\rightarrow 1^-}\chi(t)=0$, i.e., $\frac{\partial \Delta(p,t)}{\partial t}<0$.

What remains is to prove the claim. By Lemma \ref{lem:hx}(i) and Lemma \ref{lem:deltax}(i) (See Appendix), we have that the two terms are both negative on $(0,1)$. Letting
$$
\Theta(t):=2t[(1-t^{p-1})^2-(p-1)^2t^{p-2}(1-t)^2] - (p-1)(1-t)[(p-2)(1-t^p)-p(t-t^{p-1})],
$$
we have
\begin{align*}
  \Theta'(t) &= 2(2p-1)t^{2p-2} -(p^2-1)(3p-4)t^p+2p(3(p-1)^2-2)t^{p-1}\\
  &~~~~ -(p-1)^2(3p-2)t^{p-2}-2p(p-1)t+(2p^2-4p+4).\\
  \Theta''(t) &= (p-1)[4(2p-1)t^{2p-3}-p(p+1)(3p-4)t^{p-1}+2p(3(p-1)^2-2)t^{p-2}\\
  &~~~~\phantom{(p-1)} -(p-1)(3p-2)(p-2)t^{p-3}-2p].\\
  \Theta'''(t) &= (p-1)t^{p-4}[4(2p-1)(2p-3)t^p-p(p+1)(3p-4)(p-1)t^2\\
  &~~~~\phantom{(p-1)t^{p-4}} +2p(3(p-1)^2-2)(p-2)t-(p-1)(3p-2)(p-2)(p-3)]\\
  &=(p-1)t^{p-4}\Big[2(p-1)^2[6t^p-p(p+1)t^2+2p(p-2)t-(p-2)(p-3)]\\
  &~~~~\phantom{(p-1)t^{p-4}} +p(p-2)[4t^p-(p^2-1)t^2+2(p^2-2p-1)t-(p-3)(p-1)]\Big].
\end{align*}
Let $\Theta_1(t):=6t^p-p(p+1)t^2+2p(p-2)t-(p-2)(p-3)$, $\Theta_2(t):=4t^p-(p^2-1)t^2+2(p^2-2p-1)t-(p-3)(p-1)$. We first show that $t^p-1-p(t-1)\le (p-1)(1-t)^2$. This follows from the fact that
$$
\frac{d}{dt}\left(\frac{t^p-1-p(t-1)}{(1-t)^2}\right) = \frac{(p-2)(1-t^p)-p(t-t^{p-1})}{(1-t)^3}<0 \quad\text{(Lemma~\ref{lem:hx}(i), See Appendix)}.
$$
Then we have
\begin{align*}
  \Theta_1(t) &=6(t^p-1-p(t-1))-p(p+1)(1-t)^2\\
  &\le 6(p-1)(1-t)^2-p(p+1)(1-t)^2\\
  & = -(p-2)(p-3)(1-t)^2<0.\\
  \Theta_2'(t) &= 4pt^{p-1} -2(p^2-1)t + 2(p^2-2p-1).\\
  \Theta_2''(t) &= 2(p-1)t^{p-2}(2p - (p+1)t^{2-p})>0.
\end{align*}
Thus $\Theta_2'(t)<\Theta_2'(1) = 0$, which implies $\Theta_2(t)$ is decreasing and $\Theta_2(t)>\Theta_2(1)=0$. Because $\Theta_1(t)<0$ and $\Theta_2(t)>0$, we have that $\Theta'''(t)<0$. Therefore, $\Theta''(t)>\Theta''(1)=0$, which implies that $\Theta'(t)$ is increasing. Thus $\Theta'(t)<\Theta'(1)=0$, which that implies $\Theta(t)$ is decreasing, i.e., $\Theta(t)>\Theta(1)=0$. Then the claim follows directly.

(ii) When $p>2$, notice that the derivative of $\Delta(p,t)$ at $t=0$ is
$$
\lim_{t\rightarrow 0^+} \frac{\partial \Delta(p,t)}{\partial t} = \frac{\left(\frac{1}{p}\right)^{\frac{1}{p-1}}}{\frac{p-1}{p}} -\frac{p-1}{p}.
$$
When $p>6.236$, the derivative would become negative. Therefore, we could not expect that $\Delta(p,t)$ is increasing when $p>6.236$.

Instead, we are going to show that the function $(1-t)\Delta(p,t)$ is increasing. Its derivative is
$$
\left(\frac{1-t^p}{p(1-t)}\right)^{\frac{1}{p-1}}\frac{(p-1)t^p+1-pt^{p-1}}{(p-1)(1-t^p)(1-t)}-\frac{(p-1)t^{p-2}(t^p+p-1-pt)}{p(1-t^{p-1})^2}.
$$
We are going to demonstrate that this derivative is positive. Let
\begin{align*}
\Omega(t)&:=\log\left(\left(\frac{1-t^p}{p(1-t)}\right)^{\frac{1}{p-1}}\frac{(p-1)t^p+1-pt^{p-1}}{(p-1)(1-t^p)(1-t)}\right) - \log\left(\frac{(p-1)t^{p-2}(t^p+p-1-pt)}{p(1-t^{p-1})^2}\right)\\
& = \log\left(\left(\frac{1-t^p}{p(1-t)}\right)^{\frac{1}{p-1}}\right) -\log\left(\frac{(p-1)(1-t^p)}{p(1-t^{p-1})}\right) - \log\left(\frac{(p-1)t^{p-2}(1-t)(t^p+p-1-pt)}{(1-t^{p-1})((p-1)t^p+1-pt^{p-1})}\right).
\end{align*}
\begin{align*}
  \Omega'(t)&=\frac{(1-t^{p-1})^2-(p-1)^2t^{p-2}(1-t)^2}{(p-1)(1-t)(1-t^{p-1})(1-t^p)} - \frac{(p-2)-(p-1)t+t^{p-1}}{t(1-t)(1-t^{p-1})} \\
  &~~~~ + \frac{p[(1-t^{p-1})^2-(p-1)^2t^{p-2}(1-t)^2]}{((p-1)t^p+1-pt^{p-1})(t^p+p-1-pt)}\\
  &=\frac{(1-t^{p-1})^2-(p-1)^2t^{p-2}(1-t)^2}{(p-1)(1-t)(1-t^{p-1})(1-t^p)} - \frac{t^p+(p-1)-pt}{t(1-t)((p-1)t^p+1-pt^{p-1})} \\
  &~~~~ + \frac{(p-1)((p-1)t^p+1-pt^{p-1})}{t(1-t^{p-1})(t^p+(p-1)-pt)}\\
  &=\frac{p[(1-t^{p-1})^2-(p-1)^2t^{p-2}(1-t)^2]}{(p-1)(1-t)(1-t^p)((p-1)t^p+1-pt^{p-1})} \\
  &~~~~ - \frac{(p-1)((p-2)(1-t^p)-p(t-t^{p-1}))}{t(1-t^{p-1})(t^p+(p-1)-pt)}.
\end{align*}
We claim that
$$
0<\frac{p[(1-t^{p-1})^2-(p-1)^2t^{p-2}(1-t)^2]}{(p-1)(1-t^p)}<\frac{(p-2)(1-t^p)-p(t-t^{p-1})}{t}.
$$
Then
\begin{align*}
  \Omega'(t)&<\frac{(p-2)(1-t^p)-p(t-t^{p-1})}{t(1-t^{p-1})(1-t)}\left(\frac{1-t^{p-1}}{(p-1)t^p+1-pt^{p-1}} - \frac{(p-1)(1-t)}{t^p+(p-1)-pt}\right)\\
  &=\frac{(p-2)(1-t^p)-p(t-t^{p-1})}{t(1-t^{p-1})(1-t)}\left(\frac{-t[(1-t^{p-1})^2-(p-1)^2t^{p-2}(1-t)^2]}{((p-1)t^p+1-pt^{p-1})(t^p+(p-1)-pt)}\right)<0.
\end{align*}
Therefore, $\Omega(t)>\lim_{t\rightarrow 1^-}\Omega(t)=0$, i.e., the derivative of $(1-t)\Delta(p,t)$ is positive.

We only need to prove the claim. Letting
$$
\Phi(t):=pt[(1-t^{p-1})^2-(p-1)^2t^{p-2}(1-t)^2] - (p-1)(1-t^p)[(p-2)(1-t^p)-p(t-t^{p-1})],
$$
we have
\begin{align*}
  \Phi'(t) &= p[-2(p-1)(p-2)t^{2p-1} + p(2p-1)t^{2p-2}-p(p^2-1)t^p \\
  &~~~~ + 2(p-2)(p^2+p-1)t^{p-1} -p(p-1)^2t^{p-2}+p].\\
  \Phi''(t) &= p(p-1)t^{p-3}[-2(p-2)(2p-1)t^{p+1} +2p(2p-1)t^p -p^2(p+1)t^2\\
  &~~~~\phantom{p(p-1)t^{p-3}} +2(p-2)(p^2+p-1)t - p(p-1)(p-2)].
\end{align*}
Let $\Phi_1(t):=\frac{\Phi''(t)}{p(p-1)t^{p-3}} = 2(p-2)(2p-1)t^{p+1} +2p(2p-1)t^p -p^2(p+1)t^2 +2(p-2)(p^2+p-1)t - p(p-1)(p-2)$.
Then \begin{align*}
  \Phi_1'(t) &= -2(p-2)(2p-1)(p+1)t^p + 2p^2(2p-1)t^{p-1} -2p^2(p+1)t +2(p-2)(p^2+p-1);\\
  \Phi_1''(t) &= -2(p-2)(2p-1)(p+1)p t^{p-1} +2p^2(2p-1)(p-1)t^{p-2} -2p^2(p+1);\\
  \Phi_1'''(t) &= -2p(p-2)(2p-1)(p-1)t^{p-3}[(p+1)t - p].
\end{align*}
Therefore, we have that $\Phi_1''(t)$ is increasing on $(0,\frac{p}{p+1})$ and decreasing on $(\frac{p}{p+1},1)$. We have $\Phi_1''(t)\le\Phi_1''(\frac{p}{p+1})=2p^2(2p-1)\left(\frac{p}{p+1}\right)^{p-2}-2p^2(p+1)$. Letting $\Phi_2(p):=(p-2)\log\left(\frac{p}{p+1}\right) -\log\left(\frac{p+1}{2p-1}\right)$, we have
\begin{align*}
  \Phi_2'(p) &=\frac{p-2}{p} +\log(p) -\frac{p-1}{p+1} -\log(p+1) +\frac{2}{2p-1};\\
  \Phi_2''(p) &=\frac{(p-2)(8p^2+p-1)}{p^2(p+1)^2(2p-1)^2}.
\end{align*}
Therefore $\Phi_2'(p)$ is increasing on $(2,\infty)$. Along with $\lim_{p\rightarrow \infty}\Phi_2'(p)=0$, we have that $\Phi_2'(p)<0$ on $(2,\infty)$, which implies that $\Phi_2(p)<\Phi_2(2)=0$. Thus $\Phi_1''(t)<0$. Then we have that $\Phi_1'(t)$ is decreasing on $(0,1)$, which implies that $\Phi_1'(t)>\Phi_1'(1) =0$. Therefore, we have $\Phi_1(t)<\Phi_1(1)=0$, i.e., $\Phi''(t)<0$. Then we conclude that $\Phi'(t)$ is decreasing on $(0,1)$, which implies that $\Phi'(t)>\Phi(1)=0$.
Therefore,  $\Phi(t)$ is increasing on $(0,1)$ and $\Phi(t)<\Phi(1)=0$, which proves the claim.
\end{proof}

Because of Theorem \ref{prop:deltamono},
we can focus on the special case $\ell=0$. So we define
$$
\Delta(p):= \Delta(p,0)= \left(\frac{1}{p}\right)^{\frac{1}{p-1}} -\frac{p-1}{p}.
$$

\begin{figure}[H]
  \centering
  \includegraphics[scale = 0.3]{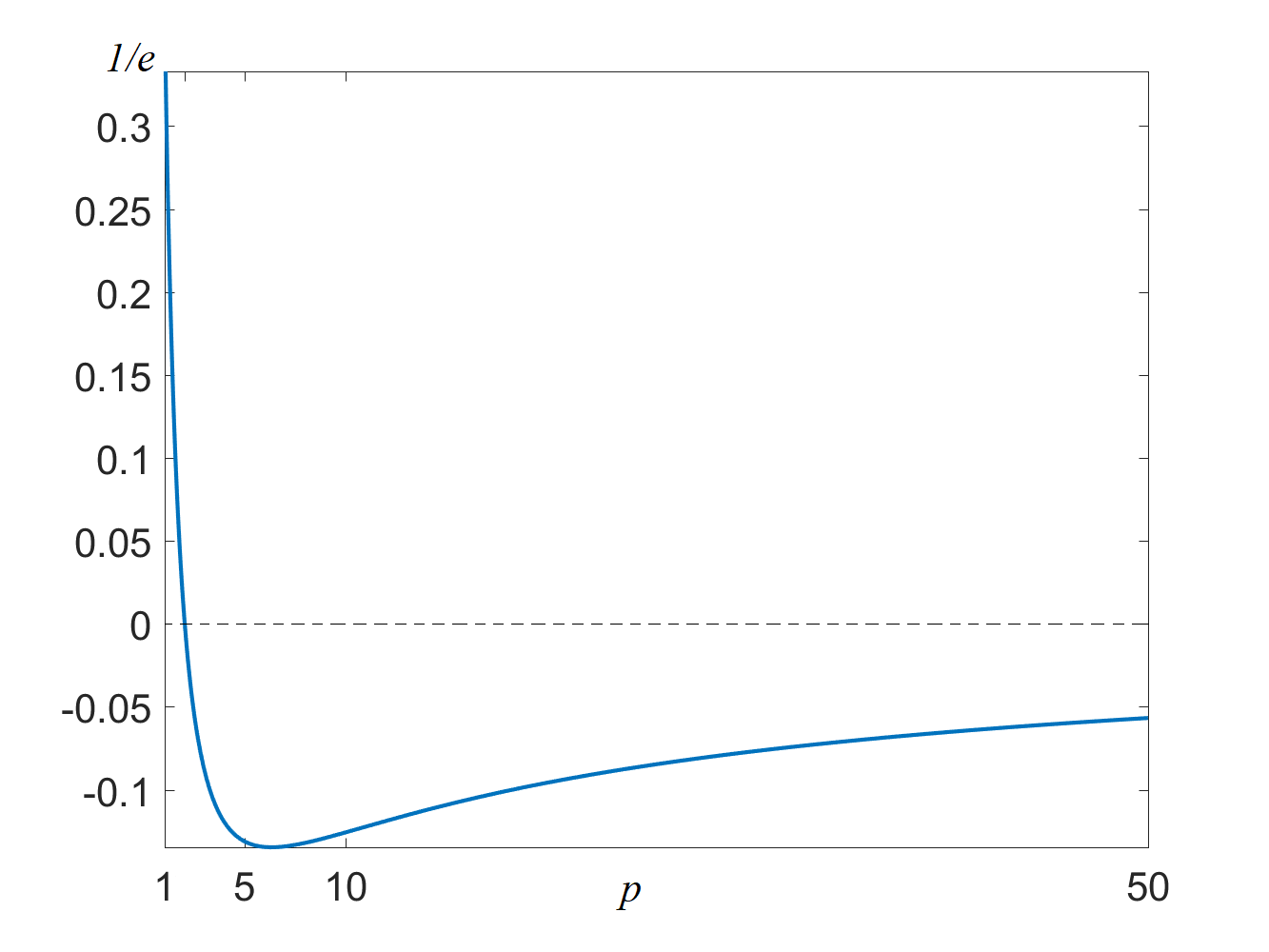}
  \caption{$\Delta(p)$}\label{fig:ub_lb_diff}
\end{figure}
\noindent
From Figure \ref{fig:ub_lb_diff}, 
we can see the behavior of $\Delta(p)$,
which is summarized in the following result.

\begin{proposition}\label{prop:deltabehavior0}
$\Delta(p)$ $(p>1)$ satisfies the following properties:
\begin{itemize}
  \item[(i)] $\Delta(p)>0$ when $1<p<2$; $\Delta(2)=0$; and $\Delta(p)<0$ when $p>2$;
  \item[(ii)] $\lim\limits_{p\rightarrow 1}\Delta(p) = e^{-1}$; $\lim\limits_{p\rightarrow \infty}\Delta(p) = 0$;
  \item[(iii)] $\Delta(p)$ is minimized at $p_0$, where $p_0\approx 6.3212$;
  \item[(iv)] $0.3679\approx e^{-1}\geq \Delta(p)\geq \Delta(p_0)\approx -0.1347$.
\end{itemize}
\end{proposition}
\begin{proof}
  (i) follows from Theorem \ref{conj:location}. For (ii),
  $$
  \lim\limits_{p\rightarrow 1}\Delta(p) = \lim\limits_{p\rightarrow 1}\exp\left\{-\frac{\log p}{p-1}\right\}=\exp\left\{-1\right\}; \lim\limits_{p\rightarrow \infty}\Delta(p) = \lim\limits_{p\rightarrow \infty}\exp\left\{-\frac{\log p}{p-1}\right\}-1 = 0.
  $$
  For (iii), we have
  \begin{align*}
  \Delta'(p) &= \left(\frac{1}{p}\right)^{\frac{1}{p-1}}\left[-\frac{1}{p(p-1)}+\frac{\log p}{(p-1)^2}\right] -\frac{1}{p^2}\\
  &=\left(\frac{1}{p}\right)^{\frac{1}{p-1}}\frac{1}{p^2}\left[-\frac{p}{p-1}+\frac{p^2\log p}{(p-1)^2}-p^{\frac{1}{p-1}}\right].
  \end{align*}
  Notice that
  \begin{align*}
  \frac{d}{dp}\left(p^{2+\frac{1}{p-1}}\Delta'(p)\right)&=\frac{p^2-1-2p\log p}{(p-1)^3} - p^{\frac{1}{p-1}}\frac{(p-1)-p\log p}{p(p-1)^2} ~>~ 0.
  \end{align*}
This follows from $p^2-1-2p\log p>0$ and $(p-1)-p\log p>0$ for $p>1$. Therefore, $p^{2+\frac{1}{p-1}}\Delta'(p)$ is increasing on $(1,\infty)$. There exists unique $p_0>1$ satisfying
$$
-\frac{p_0}{p_0-1}+\frac{p_0^2\log p_0}{(p_0-1)^2}-p_0^{\frac{1}{p_0-1}}=0,
$$
and $\Delta'(p)<0$ for  $1<p<p_0$, $\Delta'(p)>0$ for $p>p_0$, which implies that $\Delta(p_0)=\min_{p>1}\Delta(p)$. (iv) follows directly.
\end{proof}

\section{Lighter relaxations}\label{sec:lighter}

As we mentioned at the outset, an alternative key relaxation previously studied
requires that the domain of $f$ is all of $[0,u]$,
$f$ is convex on $[0,u]$, $f(0)=0$,
and $f$ is  increasing on $[0,u]$.
Assuming these properties, we recall the definition of the \emph{na\"{\i}ve relaxation}
\begin{align*}
&\invbreve{S}^0_f(\ell,u):= 
\left\{ (x,y,z) \in \mathbb{R}^3 ~:~
\left(f(\ell)-  {\scriptstyle \frac{f(u)-f(\ell)}{u-\ell}} \ell\right)z
  + {\scriptstyle \frac{f(u)-f(\ell)}{u-\ell}} x
\geq y \geq f(x),~ \right.\\
&\left. uz\geq x \geq  \ell z,~   1\geq z \geq 0
\vphantom{\scriptstyle\frac{f(u)-f(\ell)}{u-\ell}}
\right\}.
\end{align*}
For example,
convex power functions $f(x):=x^p$ on $[\ell,u]$, $\ell\geq 0$,  with $p>1$ have the required properties.
We wish to discuss a few different ways to handle functions $f$ with these properties.

\begin{itemize}[labelwidth=0pt, itemindent=!,labelindent=0pt]
\item[$\cdot$] \underline{N}a\"{i}ve \underline{R}elaxation [NR]: $\invbreve{S}^0_f(\ell,u)$
\item[$\cdot$] \underline{P}erspective \underline{R}elaxation [PR]: $\invbreve{S}^*_f(\ell,u)$
\item[$\cdot$] \underline{P}iecewise-\underline{L}inear under-est. $+$ \underline{P}erspective \underline{R}elaxation [PL+PR]: $\invbreve{U}^*_f(\bm\xi):=\invbreve{S}^*_g(\ell,u)$
\item[$\cdot$] linearly \underline{E}xtend to 0 $+$ \underline{N}a\"{i}ve \underline{R}elaxation [E+NR]: $\invbreve{S}^0_{\bar{f}}(\ell,u)$
\item[$\cdot$] \underline{P}iecewise-\underline{L}inear under-est. $+$ linearly \underline{E}xtend to 0 $+$ \underline{N}a\"{i}ve \underline{R}elaxation  [PL+E+NR]: $\invbreve{U}^0_{\bar{f}}(\bm\xi):=
    \invbreve{S}^0_{\bar{g}}(\ell,u)$
\end{itemize}

One of the main focuses of \cite{Perspec2019} was comparing NR and PR, with the idea
that PR is tighter than NR, but PR is more burdensome computationally.
So far in this work, we have extensively investigated PL+PR, again with the motivation that
 PL+PR is less burdensome than PR. Because piecewise-linearization
 requires choosing linearization points, we have put a big emphasis on
 how to do that. When $\ell>0$, a simple way to do something stronger than
 NR is with E+NR: linearly interpolate on $[0,\ell]$ before applying the
na\"{i}ve relaxation --- the strict convexity of the power function makes
this stronger than NR. Finally, again when $\ell>0$, we can consider PL+E+NR:
applying piecewise-linearization on $[\ell,u]$, linearly interpolating on $[0,\ell]$,
and then applying the na\"{i}ve relaxation.

In what follows, we focus on power functions, but the ideas could also be applied to other functions
having the required properties.

%

\subsection{PL+E+NR}\label{sec:lighter:PL+E+NR}
Defining the piecewise-linear $g$ with respect to $f$ having domain $[\ell,u]$,  we
can  extend $g$ to the function $\bar{g}$, with domain all of $[0,u]$:
 \[
 \bar{g}(x):=\left\{
         \begin{array}{ll}
           \frac{f(\ell)}{\ell} x, & x\in[0,\ell); \\
           g(x), & x\in[\ell,u].
         \end{array}
       \right.
 \]
In this way, $\bar{g}$ is a piecewise-linear increasing function on all of $[0,u]$, and is convex on $[0,u]$ as long as $f'(\ell) \geq \frac{f(\ell)}{\ell}$.
In fact, $\bar{g}$ is an under-estimator of the function that is $f$ on $[\ell,u]$ and 0 at 0.  Next
we calculate the volume of the na\"{i}ve relaxation of the piecewise-linear under-estimator $\invbreve{U}^0_{\bar{f}}(\bm\xi):=\invbreve{S}^0_{\bar{g}}(\ell,u)$, by applying \cite[Thm. 10]{Perspec2019} to $\bar{g}$.


\begin{proposition}\label{prop:naiveconvex}  Suppose that $f$ is convex and increasing on $[\ell,u]$ with $f'(\ell) \geq \frac{f(\ell)}{\ell}$.
For $\bm\xi=(\ell,\xi_1,\dots,\xi_{n-1},u)$, where $f$ is differentiable at each coordinate of $\bm\xi$, we can compute $\invbreve{U}^0_{\bar{f}}(\bm\xi)$
in $\mathcal{O}(n)$ time.
\end{proposition}

\begin{proof}
We define the $\tau_i$ and $g$ from $f$,$\ell$,$u$  as usual.
For $x\in [\ell,u]$, we have
\[
\bar{g}(x)=g(x) = g(\tau_i) + \frac{g(\tau_{i+1})-g(\tau_{i})}{\tau_{i+1}-\tau_i}(x-\tau_i), ~\forall~x\in [\tau_i,\tau_{i+1}], ~i=0,1,\dots,n.
\]
Applying \cite[Thm. 10]{Perspec2019} to $\bar{g}$, we have
\begin{align*}
\invbreve{S}^0_{\bar{g}}(\ell,u) &= \int_{g(\ell)}^{g(u)} \left(g^{-1}(y)-\frac{g^{-1}(y)^2}{2u} \right) dy  \\
& \qquad - \frac{\ell}{2}(g(u)-g(\ell)) - \frac{u-\ell}{6u}(u g(u) - \ell g(\ell)) - \frac{u-\ell}{6}(g(u) - g(\ell)) \\
&= \sum_{i=0}^n \int_{g(\tau_i)}^{g(\tau_{i+1})} \left(g^{-1}(y)-\frac{g^{-1}(y)^2}{2u} \right) dy \\
& \qquad - \frac{\ell}{2}(f(u)-f(\ell)) - \frac{u-\ell}{6u}(u f(u) - \ell f(\ell)) - \frac{u-\ell}{6}(f(u) - f(\ell)) \\
&= \sum_{i=0}^n \int_{\tau_i}^{\tau_{i+1}} \left(w-\frac{w^2}{2u} \right) \frac{g(\tau_{i+1})-g(\tau_{i})}{\tau_{i+1}-\tau_i} dw \\
& \qquad - \frac{u+2\ell}{6}(f(u)-f(\ell)) - \frac{u-\ell}{6u}(u f(u) - \ell f(\ell)) \\
&= \sum_{i=0}^n \left( \frac{\tau_{i+1}^2-\tau_i^2}{2} - \frac{\tau_{i+1}^3-\tau_i^3}{6u} \right) f'(\xi_i) \\
& \qquad - \frac{u+2\ell}{6}(f(u)-f(\ell)) - \frac{u-\ell}{6u}(u f(u) - \ell f(\ell))
\end{align*}
The result follows.
\end{proof}

Next, we consider the case of convex power functions $f(x):=x^p$ on $[\ell,u]$, with $p>1$.
To emphasize that the calculations are for power functions with exponent $p$ (>1),
we will write $\invbreve{U}^0_{\bar{p}}(\bm\xi)$ rather than $\invbreve{U}^0_{\bar{f}}(\bm\xi)$.


\begin{corollary}\label{cor:naivepoly}
For $\bm\xi=(\ell,\xi_1,\dots,\xi_{n-1},u)$, we can compute $\invbreve{U}^0_{\bar{p}}(\bm\xi)$
in $\mathcal{O}(n)$ time.
\end{corollary}

\noindent For quadratics and equally-spaced linearization points, we get a simple expression.


\begin{corollary}\label{cor:naive2equal}
For $p=2$, and the equally-spaced points $\xi_i=\ell + \frac{i}{n}(u-\ell)$, for $i=1,\dots,n-1$,
\[
\invbreve{U}^0_{\bar{2}}(\bm\xi)
=\frac{(u-\ell)^2(u^2+\ell^2)}{12u}+\frac{(u-\ell)^4}{24n^2u}.
\]
\end{corollary}

\begin{proof}
\begin{align*}
\vol(\invbreve{U}^0_{\bar{2}}(\bm\xi))
&= \sum_{i=0}^n\left(-\frac{1}{6u}(\tau_{i+1}^3-\tau_i^3)+\frac12(\tau_{i+1}^2-\tau_i^2)\right)2\xi_i\\
& \qquad - \frac{u+2\ell}{6}(u^2-\ell^2) - \frac{u-\ell}{6u}(u^3 - \ell^3) \\
&= \frac{3}{4}(u^3-\ell^3) + \frac14\sum_{i=1}^n\xi_i\xi_{i-1}(\xi_{i-1}-\xi_i) +\\
&\qquad-\frac{7}{24u}(u^4-\ell^4) -\frac{1}{12u}\sum_{i=1}^n\xi_{i-1}\xi_i(\xi_{i-1}^2-\xi_i^2)\\
& \qquad - \frac{u+2\ell}{6}(u^2-\ell^2) - \frac{u-\ell}{6u}(u^3 - \ell^3) \\
&=\frac{(u-\ell)^2(u^2+\ell^2)}{12u}+\frac{(u-\ell)^4}{24n^2u}.
\end{align*}
\end{proof}


\begin{remark}
Letting $n$ go to infinity in Corollary \ref{cor:naive2equal},
we obtain Corollary 11 of \cite{Perspec2019} with $p=2$.
\end{remark}

\subsection{E+NR}\label{sec:lighter:E+NR}
Continuing this idea, but without piecewise-linearization on its
domain  $[\ell,u]$, we can extend $f$ to the function $\bar{f}$, with domain $[0,u]$,
\[
 \bar{f}(x):=\left\{
         \begin{array}{ll}
           \frac{f(\ell)}{\ell} x, & x\in[0,\ell); \\
           f(x), & x\in[\ell,u].
         \end{array}
       \right.
 \]
 Applying the na\"{i}ve relaxation to  $\bar{f}$, we write
  $\invbreve{S}^0_{\bar{f}}(\ell,u)$.
It is clear that $\bar{g}$ (as defined above) is a lower bound on $\bar{f}$, so the na\"{i}ve relaxations associated with these functions are nested:
$\invbreve{S}^0_{\bar{f}}(\ell,u) \subset
\invbreve{U}_{\bar{f}}^0(\bm\xi):=
\invbreve{S}^0_{\bar{g}}(\ell,u)$.
We are naturally interested in how many linearization points
are sufficient to get $\vol(\invbreve{U}_{\bar{f}}^0(\bm\xi))$
to be close to $\invbreve{S}^0_{\bar{f}}(\ell,u)$.
We can give an answer to this in the case of the quadratic.
In what follows, we will write $\invbreve{U}_{\bar{2}}^0(\bm\xi)$ for
$\invbreve{U}_{\bar{f}}^0(\bm\xi)$, to emphasize the special case.

 %

\begin{proposition}\label{prop:volcomp1} For  equally-spaced points $\xi_i:=\ell + \frac{i}{n}(u-\ell)$, for $i=1,\dots,n-1$, if
\[
n > \frac{(u - \ell)^2}{\sqrt{24u\phi}},
\mbox{ then }
\vol(\invbreve{U}^0_{\bar{2}}(\bm\xi)\backslash\invbreve{S}_{\bar{2}}^0(\ell,u)) < \phi.
\]
\end{proposition}
\begin{proof}
Applying Corollary 11 of \cite{Perspec2019} with $p=2$, we find that
\[
\vol(\invbreve{S}_{\bar{2}}^0(\ell,u)) = \frac{(u-\ell)^2(u^2+\ell^2)}{12u}.
\]
As noted above, $\invbreve{S}_{\bar{2}}^0(\ell,u) \subseteq \invbreve{U}_{\bar{2}}^0(\bm\xi)$, and by Corollary \ref{cor:naive2equal},
\[
\vol(\invbreve{U}^0_{\bar{2}}(\bm\xi)\backslash\invbreve{S}_{\bar{2}}^0(\ell,u)) =\vol(\invbreve{U}^0_{\bar{2}}(\bm\xi))-\vol(\invbreve{S}_{\bar{2}}^0(\ell,u)) = \frac{(u-\ell)^4}{24n^2u}.
\]
The lower bound on $n$ to obtain $\vol(\invbreve{U}^0_{\bar{2}}(\bm\xi)\backslash\invbreve{S}_{\bar{2}}^0(\ell,u)) < \phi$ follows easily.
\end{proof}

The result above found how many linearization points are sufficient to get
the na\"{i}ve volumes of E+NR and PL+E+NR close for quadratics. We can do the
same for the volumes of PR and PL+PR.  The perspective case is especially nice because we know that choosing
equally-spaced linearization points is optimal.

\begin{proposition}\label{prop:volcomp2} For  equally-spaced points $\xi_i:=\ell + \frac{i}{n}(u-\ell)$, for $i=1,\dots,n-1$, if
\[
n > \frac{1}{6}\sqrt{\frac{(u - \ell)^3}{\phi}},
\mbox{ then }
\vol(\invbreve{U}^*_{2}(\bm\xi)\backslash\invbreve{S}_{2}^*(\ell,u)) < \phi.
\]
\end{proposition}
\begin{proof}
By Corollary \ref{cor:S2},
\[
\vol(\invbreve{S}_{2}^*(\ell,u))=\frac{(u-l)^3}{18},
\]
and by Theorem \ref{thm:p2n},
\[\vol(\invbreve{U}^*_{2}(\bm\xi)) = \frac{(u-l)^3}{18} +\frac{(u-l)^3}{36n^2}.\]

Clearly $\invbreve{S}_{2}^*(\ell,u) \subset \invbreve{U}^*_{2}(\bm\xi)$ and
\[
\vol(\invbreve{U}^*_{2}(\bm\xi)\backslash\invbreve{S}_{2}^*(\ell,u)) =\vol(\invbreve{U}^*_{2}(\bm\xi))-\vol(\invbreve{S}_{2}^*(\ell,u)) = \frac{(u-\ell)^3}{36n^2}.
\]
The lower bound on $n$ to obtain $\vol(\invbreve{U}^*_{2}(\bm\xi))\backslash\vol(\invbreve{S}_{2}^*(\ell,u)) < \phi$ follows easily.
\end{proof}

\begin{remark}
It is interesting to compare Propositions \ref{prop:volcomp1} and \ref{prop:volcomp2}.
Proposition \ref{prop:volcomp1} tells us that if we want to ``$\phi$-approximate''
 E+NR  with  PL+E+NR (i.e., using piecewise linearization), then we can do this using a certain number of equally-spaced
  linearization points, $n_1$.
 Similarly, if we want to $\phi$-approximate PR with PL+PR (i.e., using piecewise linearization), then we can do this using a certain number of
 equally-spaced linearization points, $n_2$. It is easy to check that, for \emph{all} $\phi$, we have that
 \[
 \frac{n_1}{n_2} = \sqrt{\frac{3}{2}\left(1-\frac{\ell}{u}\right)}.
 \]
 \end{remark}
 So the number of  equally-spaced  linearization points in the former case is more than in the latter case, if and only if
 $\frac{\ell}{u} < \frac{1}{3}$, and the factor $ \frac{n_1}{n_2}$ is never more than $\sqrt\frac{3}{2}\approx 1.225$.



\begin{acknowledgements}
   J. Lee was supported in part by ONR grant N00014-17-1-2296. D. Skipper was supported in part by ONR grant N00014-18-W-X00709.
   E. Speakman was supported by the Deutsche Forschungsgemeinschaft (DFG, German Research Foundation) - 314838170, GRK 2297 MathCoRe. Lee and Skipper gratefully acknowledge additional support from
   the Institute of Mathematical Optimization, Otto-von-Guericke-Universit\"{a}t, Magdeburg, Germany.
\end{acknowledgements}

\appendix  
\section*{Appendix}
\stepcounter{section}\label{sec:appendix}
\begin{lemma}\label{lem:positive}
For $x\in(0,1)\cup (1,\infty)$, $p>1$,
  $$
  x^p + (p-1) - px>0, \quad (p-1)x^p + 1 - px^{p-1}>0.
  $$
\end{lemma}
\begin{proof}
$x^p + (p-1) - px = x^p - 1 - p(x-1)>0$ because of the strict convexity of $x^p$ on $(0,\infty)$ for $p>1$. $(p-1)x^p + 1 - px^{p-1} = 1 - x^p - px^{p-1}(1-x)>0$ because of the strict convexity of $x^p$ on $(0,\infty)$ for $p>1$.
\end{proof}
\begin{lemma}\label{lem:hx}
  Letting $h(x):=(p-2)(x^p-1) - p(x^{p-1}-x)$, we have
  \begin{itemize}
    \item[(i)] if $1<p<2$, then $h(x)>0$ for $x\in(0,1)$;
    \item[(ii)] if $p>2$, then $h(x)<0$ for $x\in(0,1)$.
  \end{itemize}
\end{lemma}
\begin{proof}
  We have
  \begin{align*}
    h'(x) &= (p-2)px^{p-1} - p(p-1)x^{p-2} + p\\
    h''(x) &= (p-2)(p-1)px^{p-3}(x-1)
  \end{align*}
(i) If $1<p<2$, then $h''(x)>0$ on $(0,1)$, which implies that $h'(x)$ is increasing. Thus $h'(x)<h'(1)=0$, which implies that $h(x)$ is decreasing. Therefore, $h(x)>h(1)=0$. (ii) Similarly, we could prove that $h(x)$ is increasing and $h(x)<0$ on $(0,1)$.
\end{proof}
\begin{remark}
  Notice that $h(x) = -x^p h(1/x)$, we have $h(x)<0$ on $(1,\infty)$ when $1<p<2$, and $h(x)>0$ on $(1,\infty)$ when $p>2$.
\end{remark}
\begin{lemma}\label{lem:deltax}
  Letting $\delta(x) := (x^{p-1}-1)^2 - (p-1)^2 x^{p-2}(x-1)^2$, we have
  \begin{itemize}
    \item[(i)] if $1<p<2$, then $\delta(x)<0$ on $(0,1)\cup(1,\infty)$;
    \item[(ii)] if $p>2$, then $\delta(x)> 0$ on $(0,1)\cup(1,\infty)$.
  \end{itemize}
\end{lemma}
\begin{proof}
  Notice that $\delta(x) = x^{2p-2}\delta(1/x)$, we only need to show the results on $(0,1)$. Letting
  $$
  \varphi(x):= 1-x^{p-1} - (p-1)x^{\frac{p-2}{2}}(1-x),
  $$
  we have
  $$
  \varphi'(x) = -(p-1)x^{\frac{p-4}{2}}\left(x^{\frac{p}{2}}-1 -\frac{p}{2}(x-1)\right).
  $$
  (i) $\varphi'(x)> 0$ because of the strict concavity of $x^{p/2}$ when $1<p<2$. Along with $\varphi(1)=0$, we obtain that $\varphi(x)<0$ on $(0,1)$.
  (ii) Similarly, because of the strict convexity of $x^{p/2}$ when $p>2$, we obtain that $\varphi(x)>0$ on $(0,1)$.
\end{proof}
\begin{lemma}\label{lem:partialp}
For $x\in(0,1)\cup(1,\infty)$,
$$
\phi(x):= p(p-1)(1-x)x^{p-1}\log x + (x^{p-1}-1)(x^p-1)>0.
$$
\end{lemma}
\begin{proof}
  We have
  \begin{align*}
    \phi'(x) &= p(p-1)((p-1)x^{p-2}-px^{p-1})\log x +p(p-1)(1-x)x^{p-2} \\
    &~~~~+ (p-1)x^{p-2}(x^p-1)+px^{p-1}(x^{p-1}-1).\\
    \frac{\phi'(x)}{x^{p-2}}&=((p-1)-px)p(p-1)\log x +p(p-1)(1-x) \\
    &~~~~+ (p-1)(x^p-1) +p(x^p-x).\\
    \frac{d}{dx}\left(\frac{\phi'(x)}{x^{p-2}}\right) 
    &=-p^2(p-1)\log x +\frac{p(p-1)^2}{x} - p^3 +p(2p-1)x^{p-1}\\
    &=p^2(x^{p-1}-1-\log x^{p-1}) +p(p-1)\left(\frac{p-1+x^p-px}{x}\right).
  \end{align*}
  By Lemma \ref{lem:positive} and the inequality $t-1\ge \log t$, we have $\frac{d}{dx}\left(\frac{\phi'(x)}{x^{p-2}}\right)>0$.
  Because $\phi'(1)=0$, we have $\phi'(x)<0$ for $x\in(0,1)$ and $\phi'(x)>0$ for $x\in(1,\infty)$. Combined with $\phi(1)=0$, we obtain $\phi(x)>0$ for $x\in(0,1)\cup(1,\infty)$, which proves the lemma.
\end{proof}
\begin{proof}[Proof of Theorem \ref{thm:Newtonn}.]
  For $p>2$, we know that for $k\ge 0$,
  $$
  F(x^{k+1}) \le F(x^k) + F'(x^k) (x^{k+1} -x^k) =0,
  $$
  because of the concavity of $F_i(x)$ from Lemma \ref{lem:convexity} (ii). Along with $F(x^0)\le0$ (Proposition \ref{prop:initial_F}) and $[F'(x^k)]^{-1}\ge 0$ from Lemma \ref{lem:Jacobian} (ii), we know that $x^{k+1}\ge x^k$ for $k\ge 0$. Also by concavity, we have
  $$
  0\le F(u\mathbf{1}) - F(x^k) \le F'(x^k)(u\mathbf{1}-x^k),
  $$
  which implies $x^k\le u\mathbf{1}$ because $[F'(x^k)]^{-1}$ is nonnegative. Therefore the increasing bounded sequence $\{x^k\}$ has a limit $x^*=\lim_{k\rightarrow \infty} x^k$ and $F(x^*) = 0$.

  For $1<p<2$, similarly, we know that for $k\ge 0$,
  $$
  F(x^{k+1}) \ge F(x^k) + F'(x^k) (x^{k+1} -x^k) =0,
  $$
  because of the convexity of $F_i(x)$ from Lemma \ref{lem:convexity} (i). Along with $F(x^0)\ge0$ (Proposition \ref{prop:initial_F}), we know that $[F'(x^k)]^{-1}\ge 0$ from Lemma \ref{lem:Jacobian} (i). we know that $x^{k+1}\le x^k$ for $k\ge 0$. Also by convexity, we have
  $$
  0\ge F(\ell\mathbf{1}) - F(x^k) \ge F'(x^k)(\ell\mathbf{1}-x^k),
  $$
  which implies $x^k\ge \ell\mathbf{1}$ because $[F'(x^k)]^{-1}$ is nonnegative. Therefore the decreasing bounded sequence $\{x^k\}$ has a limit $x^*=\lim_{k\rightarrow \infty} x^k$ and $F(x^*) = 0$.
\end{proof}


\bibliographystyle{spmpsci}
\bibliography{perspecbib}

\end{document}